\documentclass[11pt]{amsart}

\usepackage{amsmath,amsthm,amsfonts,amscd,flafter,epsf,graphicx,color, rlepsf}
\usepackage[all,knot,arc]{xy}

\newtheorem {theorem}{Theorem}[section]

\newtheorem {lemma}[theorem]{Lemma}

\newtheorem {claim}[theorem]{Claim}
\newtheorem {proposition}[theorem]{Proposition}

\newtheorem {remark}[theorem]{Remark}

\newcommand{\R}{\ensuremath{\mathbb{R}}}

\def\dfn#1{{\em #1}}

\newcommand\tb{{\rm {tb}}}
\newcommand\maxtb{{\overline{\rm {tb}}}}

\newcommand\rot{{\rm {rot}}}
\newcommand\ssl{{\rm {sl}}}
\newcommand\K{\mathcal{K}}

\newcommand\LL{\mathcal{L}}

\newcommand\Bin{B_\textrm{in}}
\newcommand\Bout{B_\textrm{out}}
\newcommand\Din{D_\textrm{in}}
\newcommand\Dout{D_\textrm{out}}
\newcommand\xist{\xi_\textrm{st}}

\begin{document}

\title[{Twist knots}]
{Legendrian and transverse twist knots}

\author[John B.\ Etnyre]{John B.\ Etnyre}
\address{
    School of Mathematics,
    Georgia Institute of Technology,
    686 Cherry St.,
    Atlanta, GA  30332}
\email{etnyre@math.gatech.edu}
\urladdr{http://math.gatech.edu/\char126 etnyre}

\author[Lenhard L.\ Ng]{Lenhard L.\ Ng}
\address{
Mathematics Department,
Duke University,
Durham, NC 27708}
\email{ng@math.duke.edu}
\urladdr{http://www.math.duke.edu/\char126 ng}

\author[Vera V\'ertesi]{Vera V\'ertesi}
\address{
Department of Mathematics,
Massachusetts Institute of Technology,
77 Massachusetts Avenue,
Cambridge, MA 02139
}
\email{vertesi@math.mit.edu}
\thanks{}

\begin{abstract}
In 1997, Chekanov gave the first example of a
  Legendrian nonsimple knot
type: the $m(5_2)$ knot. Epstein, Fuchs, and Meyer extended his result by
showing that there are at least $n$ different Legendrian representatives with maximal Thurston--Bennequin number of the twist knot $K_{-2n}$ with crossing number $2n+1$. In
this paper we give a complete classification of Legendrian and
transverse representatives of
twist knots. In particular, we show that $K_{-2n}$ has
exactly $\lceil\frac{n^2}2\rceil$ Legendrian representatives with
maximal Thurston--Bennequin number, and $\lceil\frac{n}{2}\rceil$
transverse representatives with maximal self-linking number. Our
techniques include convex surface theory, Legendrian ruling invariants, and Heegaard Floer homology.
\end{abstract}

\maketitle

\section{Introduction}
Throughout this paper, we consider Legendrian and transverse knots in
$\R^3$ with the standard contact structure $\xist = \ker(dz-y\,dx).$

A twist knot is a twisted Whitehead double of the unknot, specifically, any knot $K=K_m$ of the type shown in Figure~\ref{fig:twistn}.
\begin{figure}[ht]
  \relabelbox \small {
  \centerline{\epsfbox{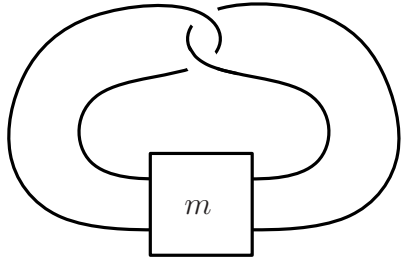}}}
  \relabel{m}{{\large $m$}}
  \endrelabelbox
        \caption{The twist knot $K_{m}$; the box contains $m$ right-handed half twists if $m\geq 0,$ and $|m|$ left-handed half twists if $m<0.$}
        \label{fig:twistn}
\end{figure}
Twist knots have long been an important class of knots to consider, particularly in contact geometry.  If Legendrian knots in a given topological knot type are determined up to Legendrian isotopy by their classical invariants, namely their Thurston--Bennequin and rotation numbers, then the knot type is said to be Legendrian simple; otherwise it is Legendrian nonsimple. While there is no reason to believe all knot types should be Legendrian simple, it has historically been difficult to prove otherwise.  Chekanov \cite{Chekanov02} and, independently, Eliashberg \cite{Eliashberg98} developed invariants of Legendrian knots that show that $K_{-4}=m(5_2)$ has Legendrian representatives that are not determined by their classical invariants, providing the first example of a Legendrian nonsimple knot. Shortly thereafter, Epstein, Fuchs, and Meyer \cite{EpsteinFuchsMeyer01} generalized the result of Chekanov and Eliashberg to show that $K_m$ is Legendrian nonsimple for all $m\leq -4$,
 and in fact that these knot types contain an arbitrarily large number of Legendrian knots with the same classical invariants. Again these were the first such examples.

One can also ask if a knot is transversely simple, that is, are transverse knots in that knot type determined by their self-linking number? It is more difficult to prove transverse nonsimplicity than Legendrian nonsimplicity. In particular, there are knot types that are Legendrian nonsimple but transversely simple \cite{EtnyreHonda03}, whereas any transversely nonsimple knot must be Legendrian nonsimple as well.  The first examples of transversely nonsimple knots were produced in 2005--6 by Birman and Menasco \cite{BirmanMenasco06II}, and Etnyre and Honda \cite{EtnyreHonda05}. It has long been suspected that some twist knots are transversely nonsimple, and this was proven very recently by Ozsv\'ath and Stipsicz \cite{OzsvathStipsicz2008Pre} using the transverse invariant in Heegaard Floer homology from \cite{LOSSz}.

Although twist knots have long supplied a useful test case for new
Legendrian invariants, such as contact homology and Legendrian
Heegaard Floer invariants (cf.\ the work of Epstein--Fuchs--Meyer and Ozsv\'ath--Stipsicz above), a complete classification of Legendrian and transverse twist knots has been elusive.
In this paper, we establish this classification and in particular identify which twist knots are Legendrian and transversely nonsimple. As a byproduct, we obtain a complete classification of an infinite family of transversely nonsimple knot types. This is one of the first (Legendrian or transversely) nonsimple families where a classification is known; see also \cite{Tosun2009Pre}.

\begin{theorem}[Classification of Legendrian twist knots] \label{thm:mtnrange}
Let $K=K_m$ be the twist knot of Figure \ref{fig:twistn}, with $m$ half twists. We discard the case $m=-1$, which is the unknot.
\begin{enumerate}
\item \label{item1a}
 For $m \geq -2$ even, there is a unique representative of $K_m$ with maximal Thurston--Bennequin number, $\tb=-m-1$. This representative has rotation number $\rot=0$,  and all other Legendrian knots of type $K_m$ destabilize to the one with maximal Thurston--Bennequin number.
\item \label{item1b}
 For $m \geq 1$ odd, there are exactly two representatives with maximal Thurston--Bennequin number, $\tb=-m-5.$ These representatives are distinguished by their rotation numbers, $\rot=\pm 1$, and a negative stabilization of the $\rot=1$ knot is isotopic to a positive stabilization of the $\rot=-1$ knot. All other Legendrian knots destabilize to at least one of these two.
\item\label{item2} For $m\leq -3$ odd, $K_m$ has $-\frac{m+1}2$ Legendrian representatives with $(\tb,\rot)=(-3,0).$ All other Legendrian knots destabilize to one of these. After any positive number of stabilizations (with a fixed number of positive and negative stabilizations), these $-\frac{m+1}{2}$ representatives all become isotopic.
\item\label{item3} For $m\leq -2$ even with $m=-2n$, $K_{m}$ has $\lceil\frac{n^2}2\rceil$ different Legendrian representations with $(\tb,\rot)=(1,0)$. All other Legendrian knots destabilize to one of these. These Legendrian knots fall into $\lceil\frac{n}{2}\rceil$ different Legendrian isotopy classes after any given positive number of positive stabilizations, and $\lceil\frac{n}{2}\rceil$ different Legendrian isotopy classes after any given positive number of negative stabilizations. After at least one positive and one negative stabilization (with a fixed number of each), the knots all become Legendrian isotopic.
\end{enumerate}
In particular, $K_m$ is Legendrian simple if and only if $m \geq -3$.
\end{theorem}
The content of Theorem~\ref{thm:mtnrange} is depicted in the Legendrian mountain ranges in Figures~\ref{fig:mainthm1} and~\ref{fig:mainthm3}. The Legendrian representatives of $K_m$ with maximal Thurston--Bennequin number will be given in Section~\ref{sec:sketch}. Note that the cases $-3\leq m\leq 2$ in Theorem~\ref{thm:mtnrange} were already known by the classification of Legendrian unknots by Eliashberg and Fraser \cite{EliashbergFraser99} and Legendrian torus knots and the figure eight knot by Etnyre and Honda \cite{EtnyreHonda01b}.

\begin{figure}
\centering
{\small
$
\vcenter{\xymatrix  @dr {
1 \ar[r]^{+} \ar[d]_{-} & 1 \ar[r]^{+} \ar[d]_{-} & 1 \ar@{.}[r] \ar@{.}[d]& \\
1 \ar[r]^{+} \ar[d]_{-} & 1 \ar@{.}[r] \ar@{.}[d]  & &\\
1 \ar@{.}[r] \ar@{.}[d] & & & \\
&&&}}
\vcenter{\xymatrix  @dr {
& 1 \ar[r]^{+} \ar[d]_{-} & 1 \ar[r]^{+} \ar[d]_{-} & 1 \ar@{.}[r] \ar@{.}[d]& \\
1 \ar[r]^{+} \ar[d]_{-} & 1 \ar[r]^{+} \ar[d]_{-}& 1 \ar@{.}[r] \ar@{.}[d]  & &\\
1 \ar[r]^{+} \ar[d]_{-} &1 \ar@{.}[r] \ar@{.}[d] & & & \\
1 \ar@{.}[r] \ar@{.}[d]&&&&\\
&&&&}}
$
}
\vspace{-3cm}
\caption{Schematic Legendrian mountain range for $K_{2n}$ ($n\ge-1$),
  left, and $K_{2n-1}$ ($n\ge1$), right.  Rotation number is plotted
  in the horizontal direction, Thurston--Bennequin number in the
  vertical direction. The numbers represent the number of Legendrian
  representatives for a particular $(\tb,\rot)$ (here, all numbers are
  $1$ since
  these knot types are Legendrian simple), and the signed arrows
  represent positive and negative stabilization.}
\label{fig:mainthm1}
\end{figure}
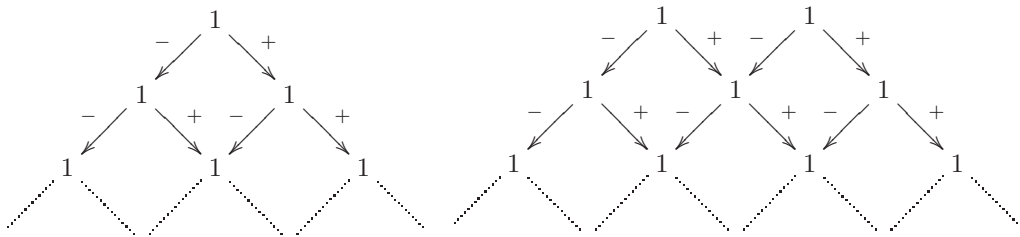

\begin{figure}
\centering
{\small
$
\vcenter{\xymatrix  @dr {
n \ar[r]^{+} \ar[d]_{-} & 1 \ar[r]^{+} \ar[d]_{-} & 1 \ar[r]^{+} \ar[d]_{-} & 1 \ar@{.}[r] \ar@{.}[d]& \\
1 \ar[r]^{+} \ar[d]_{-} & 1 \ar[r]^{+} \ar[d]_{-} & 1 \ar@{.}[r] \ar@{.}[d]  & &\\
1 \ar[r]^{+} \ar[d]_{-} & 1 \ar@{.}[r] \ar@{.}[d] & & & \\
1 \ar@{.}[r] \ar@{.}[d] & & & & \\
&&&&&}}
\vcenter{\xymatrix  @dr {
\lceil \frac{n^2}2\rceil \ar[r]^{+} \ar[d]_{-} & \lceil \frac{n}2\rceil \ar[r]^{+} \ar[d]_{-} & \lceil \frac{n}2\rceil \ar[r]^{+} \ar[d]_{-} & \lceil \frac{n}2\rceil \ar@{.}[r] \ar@{.}[d]& \\
\lceil \frac{n}2\rceil \ar[r]^{+} \ar[d]_{-} & 1 \ar[r]^{+} \ar[d]_{-} & 1 \ar@{.}[r] \ar@{.}[d]  & &\\
\lceil \frac{n}2\rceil \ar[r]^{+} \ar[d]_{-} & 1 \ar@{.}[r] \ar@{.}[d] & & & \\
\lceil \frac{n}2\rceil \ar@{.}[r] \ar@{.}[d] & & & & \\
&&&&&}}
$
}
\vspace{-3cm}
\caption{Legendrian mountain range for $K_{-2n-1}$ ($n\ge1$), left, and  $K_{-2n}$ ($n\ge1$), right.}
\label{fig:mainthm3}
\end{figure}
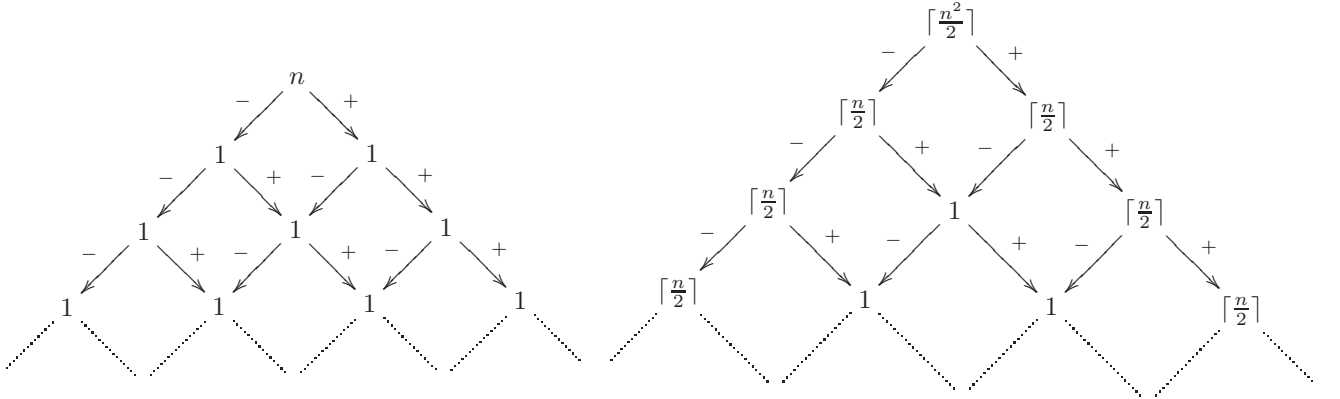

\begin{theorem}[Classification of transverse twist knots] \label{main:transverse}
Let $K=K_m$ be the twist knot of Figure \ref{fig:twistn}, with $m$ half twists.
\begin{enumerate}
\item  If $m$ is even and $m\geq -2$ or $m$ is odd, then $K_m$ is transversely simple. Moreover, the transverse representative of $K_m$ with maximal self-linking number has $\ssl=-m-1$ if $m \geq -2$ is even, $\ssl=-m-4$ if $m>-1$ is odd, $\ssl=-1$ if $m=-1$, and $\ssl=-3$ if $m<-1$ is odd.
\item If $m \leq -4$ is even with $m=-2n$, then $K_m$ is transversely nonsimple. There are $\lceil\frac{n}2\rceil$ distinct transverse representatives of $K_m$ with maximal self-linking number $\ssl=1$. Any two of these become transversely isotopic after a single stabilization, and all other transverse representatives of $K_m$ destabilize to one of these.
\end{enumerate}
\end{theorem}

We prove these classification theorems by using convex surface
techniques, along the lines of the recipe described in
\cite{EtnyreHonda01b}, to produce an exhaustive list of all
nondestabilizable Legendrian twist knots. This is the most technically
difficult part of the proof and is deferred until
Section~\ref{sec:norm}. Given this list, we use the Legendrian ruling
invariants of Chekanov--Pushkar and Fuchs, along with the
aforementioned result of Ozsv\'ath--Stipsicz, to distinguish
nonisotopic classes of Legendrian and transverse twist knots; this is
done in Section~\ref{sec:sketch}. We begin with a review of some
necessary background in Section~\ref{sec:background}.

\subsection*{Acknowledgments}

This paper was initiated by discussions between the last two authors
at the workshop ``Legendrian and transverse knots'' sponsored by the
American Institute of Mathematics in September 2008. We also
 thank Ko Honda and Andr\'as Stipsicz for helpful discussions and
 Whitney George and the referee for valuable comments on the first draft of the paper.
 JBE was partially supported by NSF grant DMS-0804820.
LLN was partially supported
by NSF grant DMS-0706777 and NSF CAREER grant DMS-0846346. VV was
partially supported by OTKA grants 49449 and 67867, and ``Magyar \'Allami \"
Oszt\"ond\'\i j''.

\section{Background and Preliminary Results}\label{sec:background}

In this section we recall some basic facts about convex surfaces and bypasses, as well as ruling invariants of Legendrian knots.

\subsection{Convex surfaces and bypasses}
Convex surfaces are the primary tool we will use in this paper. We assume the reader is familiar with this theory at the level found in \cite{EtnyreHonda01b, Giroux91, Honda00a}. For the convenience of the reader and to clarify various orientation issues we will briefly recall some of the facts about convex surfaces most germane to the proofs below, but for the basic definitions and results the reader is referred to the above references.

Recall that if $\Sigma$ is a convex surface and $\alpha$ a Legendrian arc in $\Sigma$ that intersects the dividing curves $\Gamma_\Sigma$
in 3 points $p_1,p_2,p_3$ (where $p_1, p_3$ are the endpoints of the arc), then a \dfn{bypass for $\Sigma$ (along $\alpha$)} is a
convex disk $D$ with Legendrian boundary such that
\begin{enumerate}
\item $D\cap \Sigma=\alpha,$
\item $\tb(\partial D)=-1,$
\item $\partial D= \alpha\cup \beta,$
\item $\alpha\cap\beta=\{p_1,p_3\}$ are corners of $D$ and elliptic singularities of $D_\xi.$
\end{enumerate}
The most basic property of bypasses is how a convex surface changes when pushed across a bypass.
\begin{theorem}[Honda 
\cite{Honda00a}]
Let $\Sigma$ be a convex surface and $D$ a bypass for $\Sigma$ along $\alpha\subset\Sigma.$ Inside any open neighborhood of $\Sigma\cup D$ there
is a (one sided) neighborhood $N=\Sigma\times[0,1]$ of $\Sigma\cup D$  with $\Sigma=\Sigma\times\{0\}$
(if $\Sigma$ is oriented, orient $N$ so that $\Sigma=-\Sigma\times\{0\}$ as oriented manifolds)
such that $\Gamma_\Sigma$ is related to $\Gamma_{\Sigma\times\{1\}}$ as shown in Figure~\ref{fig:bypassattach}.
\end{theorem}
\begin{figure}[ht]
  \relabelbox \small {
  \centerline{\epsfbox{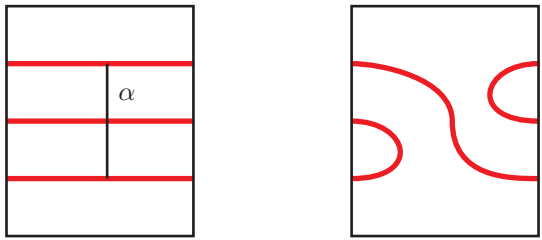}}}
  \relabel {a}{$\alpha$}
  \endrelabelbox
        \caption{Result of a bypass attachment: original surface $\Sigma$ with attaching arc $\alpha,$ left; the surface $\Sigma'=\Sigma\times\{1\}$, right. The dividing curves $\Gamma_\Sigma$ and $\Gamma_{\Sigma'}$ are shown as thicker curves.}
        \label{fig:bypassattach}
\end{figure}
In the above discussion the bypass is said to be attached from the front. To attach a bypass from the back one needs to change the orientation of the interval $[0,1]$ in the above theorem and mirror Figure \ref{fig:bypassattach}.

If $\Sigma$ and $\Sigma'$ are two convex surfaces, $\partial \Sigma'$ is a Legendrian curve contained in $\Sigma,$ and $\Sigma\cap \Sigma'=\partial \Sigma'$, then if $\Sigma'$ has a boundary parallel dividing curve (and there are other dividing curves on $\Sigma'$) then one can always find a bypass for $\Sigma$ contained in $\Sigma'$ (and containing the boundary parallel dividing curve). This is a simple application of the Legendrian realization principle \cite{Kanda}. It is useful to be able to find bypasses in other ways too. For this we have the notion of bypass rotation.

\begin{lemma}[Honda, Kazez, and Matic 
\cite{HondaKazezMatic07}]\label{moveleft}
Suppose $\Sigma$ is a convex surface containing a disk $D$ such that $D\cap \Gamma_\Sigma$ is as
shown in Figure~\ref{fig:moveright}. Also suppose $\delta$ and $\delta'$ are as shown in the figure.
If there is a bypass for $\Sigma$ attached along $\delta$ from the front side of the diagram, then there is a bypass for
$\Sigma$ attached along $\delta'$ from the front.
\end{lemma}
\begin{figure}[ht]
  \relabelbox \small {\centerline {\epsfbox{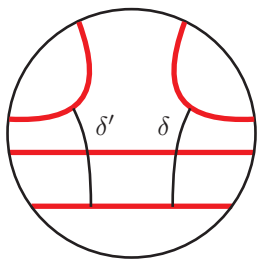}}} \relabel
        {d}{$\delta$} \relabel {2}{$\delta'$} \endrelabelbox
        \caption{If there is a bypass for $\delta$ then there is one for $\delta'$ as well.}
        \label{fig:moveright}
\end{figure}

We end our brief review of convex surfaces by describing how  two convex surfaces that come together along a Legendrian circle in their boundary can be made into a single convex surface by rounding their corners.
\begin{lemma}[Honda 
\cite{Honda00a}]\label{edgerounding}
  Suppose that $\Sigma$ and $\Sigma'$ are convex surfaces with dividing curves $\Gamma$ and
  $\Gamma'$ respectively, and $\partial \Sigma'=\partial \Sigma$ is Legendrian. Model $\Sigma$ and
  $\Sigma'$ in $\R^3$ by $\Sigma=\{(x,y,z):x=0, y\geq 0\}$ and $\Sigma'=\{(x,y,z):y=0, x\geq 0\}$. Then we may form a
  surface $\Sigma''$ from $S=\Sigma\cup \Sigma'$ by replacing $S$ in a small
  neighborhood $N$ of $\partial \Sigma$ (thought of as the $z$-axis) with the intersection of
  $N$ with $\{(x,y,z): (x-\delta)^2 +(y-\delta)^2=\delta^2\}$. For a suitably chosen $\delta,$
  $\Sigma''$ will be a smooth surface (actually just $C^1$, but it can then be smoothed by a
  $C^1$ small isotopy which can easily be seen not to change the characteristic foliation) with
  dividing curve as shown in Figure~\ref{round}.
\end{lemma}
\begin{figure}[ht]
  {
  \centerline{\epsfbox{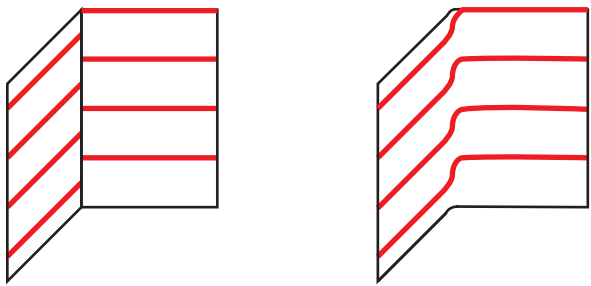}}}
  \caption{Rounding a corner between two convex surfaces. On the left, $\Sigma\cup\Sigma'$; on the right, $\Sigma''$.}
  \label{round}
\end{figure}
In this lemma, rounding a corner causes the dividing curves on the two surfaces to connect
up as follows: moving from $\Sigma$ to $\Sigma'$, the dividing curves move up (down) if $\Sigma'$ is
to the right (left) of $\Sigma.$

\subsection{Ruling invariants}
\label{subsec:ruling}

In order to distinguish between Legendrian isotopy classes of twist
knots in Section~\ref{sec:sketch}, we use invariants of Legendrian
knots in standard contact $\mathbb{R}^3$ known as the $\rho$-graded ruling
invariants, as introduced by Chekanov--Pushkar
\cite{ChekanovPushkar05} and Fuchs \cite{Fuchs}. Here we very
briefly recall the relevant definitions and results; for further details, see, e.g., the above papers or \cite{EtnyreLegendriansurvey}.

Given the front ($xz$) projection of a Legendrian knot in
$\R^3$, a \textit{ruling} is a one-to-one correspondence between left and right
cusps, along with a decomposition of the front as a union of pairs
of paths beginning at a left cusp and ending at the corresponding
right cusp, satisfying the following conditions:
\begin{itemize}
\item
all paths are smooth except possibly at double points (crossings) in
the front, and never change direction with respect to $x$ coordinate;
\item
the two paths for a particular pair of cusps do not intersect except
at the two cusp endpoints;
\item
any two arbitrary paths intersect at most at cusps and crossings;
\item
at a crossing where two paths (which must necessarily have different endpoints) intersect and one lies entirely above
the other (such a crossing is a \textit{switch}), the two paths and
their companion paths must be arranged locally as in Figure~\ref{fig:switches}.
\end{itemize}
See Figure~\ref{fig:ruling-maxtb} for examples of rulings; note that a ruling is uniquely determined by its switches, and can be thought of as a ``partial $0$-resolution'' of the front.

\begin{figure}[ht]
\centerline{
\includegraphics[width=0.8\textwidth]{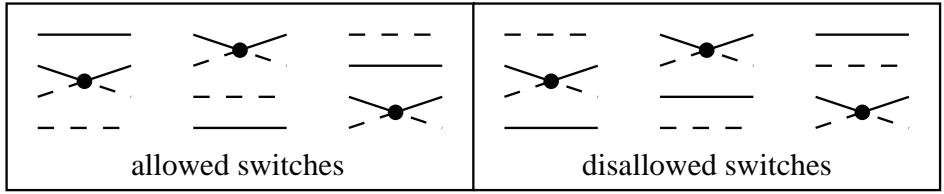}
}
\caption{Allowed and disallowed switches in a ruling. In each diagram, the two solid arcs are paired together (i.e., share cusp endpoints), as are the two dashed arcs. Other pairs of arcs, which may be present, are not shown.}
\label{fig:switches}
\end{figure}

\begin{figure}
\centerline{
\includegraphics[width=0.8\textwidth]{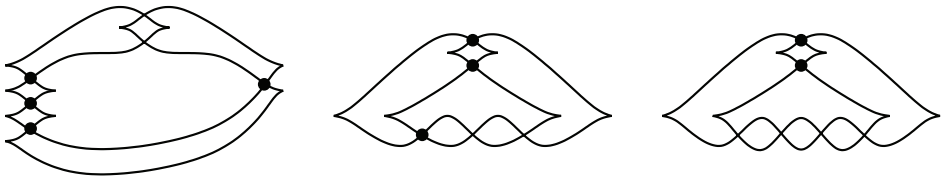}
}
\caption{
Rulings of Legendrian versions of the twist knots $K_{-4}$, $K_3$, and $K_4$. Dots indicate switches.
}
\label{fig:ruling-maxtb}
\end{figure}

One can refine the concept of a ruling by considering Maslov degrees. Removing the $2c$ left and right cusps from a front (not necessarily with a ruling) yields $2c$ arcs, each connecting a left cusp to a right cusp. If $\rot$ is the rotation number of the front, then we can assign integers (\textit{Maslov numbers}) mod $2\,\rot$ to each of these arcs so that at each cusp, the upper arc (with higher $z$ coordinate) has Maslov number $1$ greater than the lower arc; for a connected front, these numbers are well-defined up to adding a constant to all arcs. To each crossing in the front, we can define the \textit{Maslov degree} to be the Maslov number of the strand with more negative slope minus the Maslov number of the strand with more positive slope. Finally, if $\rho$ is any integer dividing $2\,\rot$, then we say that a ruling of the front is \textit{$\rho$-graded} if all switches have Maslov degree divisible by $\rho$. In particular, a $1$-graded ruling (also known as an ung
 raded ruling)
  is a ruling with no condition on the switches.

\begin{proposition}[Chekanov and Pushkar \cite{ChekanovPushkar05}] \label{prop:ChP}
Let $\K$ be a Legendrian knot with rotation number $\rot(K)$. For any $\rho$ dividing $2\,\rot(\K)$, the number of $\rho$-graded rulings of the front of $\K$ is an invariant of the Legendrian isotopy class of $\K$.
\end{proposition}

The existence of rulings is closely related to the maximal Thurston--Bennequin number of a knot.

\begin{proposition}[Rutherford \cite{Rutherford06}] \label{prop:rutherford}
If a Legendrian knot $\K$ admits an ungraded ruling, then it maximizes Thurston--Bennequin number within its topological class.
\end{proposition}

For twist knots, Proposition~\ref{prop:rutherford} allows easy calculation of the maximal value of $\tb$; we note that the following result can also be derived from the more general calculation for two-bridge knots from \cite{Ng2bridge}.

\begin{proposition} \label{prop:tbbound}
The maximal Thurston--Bennequin number for $K_m$ is:
\[
\overline{\tb}(K_m) =
\begin{cases}
-m-1 & m \geq 0 \text{ even} \\
-m-5 & m \geq 1 \text{ odd} \\
-1 & m=-1 \\
1 & m \leq -2 \text{ even} \\
-3 & m \leq -3 \text{ odd}.
\end{cases}
\]
\end{proposition}

\begin{proof}
Figure~\ref{fig:ruling-maxtb} shows ungraded rulings for Legendrian forms of $K_{-4}$, $K_3$, and $K_4$; these have obvious generalizations to Legendrian knots of type $K_m$ for $m \leq -2$, $m \geq 1$ odd, and $m \geq 0$ even, respectively, each of which has an ungraded ruling. It follows from Proposition~\ref{prop:rutherford} that each of these knots maximizes $\tb$. Easy calculations of Thurston--Bennequin numbers for each case (along with the fact that $K_{-1}$ is the unknot) yield the proposition.
\end{proof}

\section{The Classification of Legendrian Twist Knots}\label{sec:sketch}

In this section we will classify Legendrian and transverse twist knots by proving Theorems~\ref{thm:mtnrange} and \ref{main:transverse}. We begin with several preliminary results that will be proved in Section~\ref{sec:norm}.
\begin{theorem}\label{thm:frontneg}
For $m\leq -2$,
any Legendrian representative of $K=K_{m}$  with maximal $\tb$ is Legendrian isotopic to some Legendrian knot whose front projection is of the form depicted in Figure~\ref{fig:legtwistn}, where the rectangle contains $|m+2|$ negative half twists each of which is of type $Z$ or $S$.
\end{theorem}

\begin{figure}[ht]
  \relabelbox \small {
  \centerline{\epsfbox{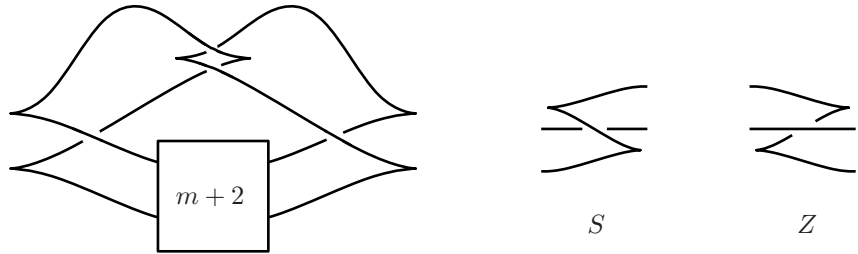}}}
  \relabel{m}{{$m+2$}}
  \relabel{s}{$S$}
  \relabel{z}{$Z$}
  \endrelabelbox
\caption{A front projection for  $K_{m}$ for $m\leq -2,$ and half twists of type $S$ and $Z.$}
\label{fig:legtwistn}
\end{figure}

\begin{theorem}\label{thm:frontpos}
For $m\geq 0$, any Legendrian representative of $K=K_{m}$  with maximal $\tb$ is Legendrian isotopic to the Legendrian knot with front projection depicted in Figure~\ref{fig:mpos}, where the rectangle contains $m$ positive half twists each of which is of type $X.$
\end{theorem}
\begin{figure}[ht]
  \relabelbox \small {
  \centerline{\epsfbox{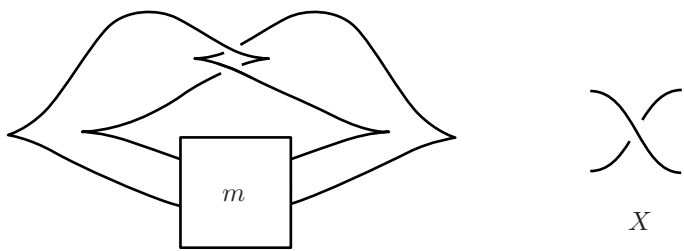}}}
  \adjustrelabel <.125in,0in> {m}{$m$}
  \relabel{x}{$X$}
  \endrelabelbox
\caption{A front projection for  $K_{m}$ for $m\geq 0,$ and crossings of type $X.$}
\label{fig:mpos}
\end{figure}

The techniques developed for the proof of the above theorems also give the following result.

\begin{theorem}\label{thm:destab}
Let $\K$ be a Legendrian representative of the twist knot $K_m.$
Whenever $\tb (\K) <\maxtb (K_m)$ then $\K$ destabilizes.
\end{theorem}

We will see below that Theorems~\ref{thm:frontpos} and~\ref{thm:destab} establish Items~\eqref{item1a} and~\eqref{item1b} in Theorem~\ref{thm:mtnrange}, the classification of Legendrian $K_m$ for $m \geq -2$.
To classify Legendrian $K_m$ for $m \leq -3$, we need to distinguish between the distinct representatives of $K_m$ with maximal Thurston--Bennequin number and understand when they become the same under stabilization.

We begin by considering $K_m$ when $m\leq -4$ is even. According to Theorem~\ref{thm:frontneg}, we can represent each of the maximal-tb representatives of $K_{-2n}$ by a
length $2n-2$ word in the letters $Z^+,Z^-,S^+,S^-$, where these letters
represent the Legendrian half-twists shown in
Figure~\ref{fig:zs-example} and letters
must alternate in sign. Given such a word $w$, let
$z^+(w),z^-(w),s^+(w),s^-(w)$ denote the number of $Z^+,Z^-,S^+,S^-$
in $w$, respectively, and note that $z^+(w)+s^+(w)=z^-(w)+s^-(w)=n-1.$
\begin{figure}[ht]
  \relabelbox \small {
  \centerline{\epsfbox{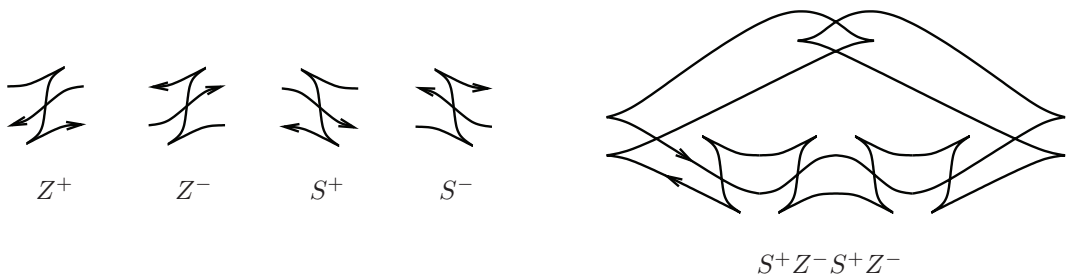}}}
  \relabel{a}{{$Z^+$}}
  \relabel{b}{$Z^-$}
  \relabel{c}{$S^+$}
  \relabel{d}{{$S^-$}}
  \relabel{e}{$S^+Z^-S^+Z^-$}
  \endrelabelbox
\caption{Denoting a maximal-tb twist knot by a word in $Z$'s and $S$'s.}
\label{fig:zs-example}
\end{figure}

\begin{lemma}\label{lem:legisotopic1}
Two words of length $2n$ with the same $z^+,z^-,s^+,s^-$ correspond to
Legendrian-isotopic knots.
\end{lemma}

\begin{figure}
\centering

\includegraphics
{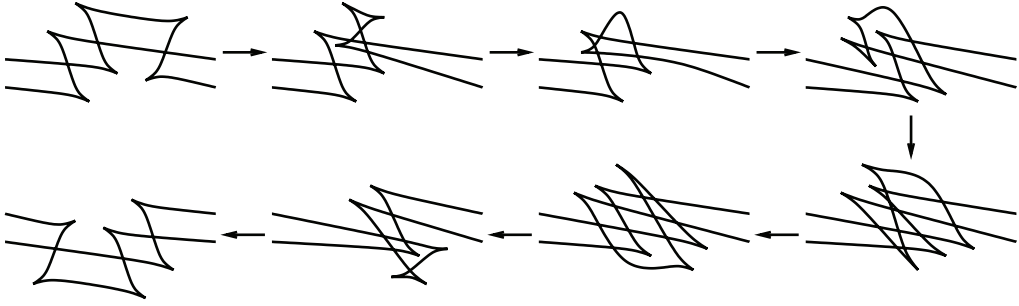}

\caption{
Legendrian isotopy between $SSZ$ and $ZSS$.
}
\label{fig:ssz}
\end{figure}

\begin{proof}
A local computation (Figure~\ref{fig:ssz}) shows that $S^\pm S^\mp
Z^\pm$ and $Z^\pm S^\mp S^\pm$ are Legendrian isotopic as Legendrian
tangles. (Alternately, the fact that these are Legendrian isotopic
follows from the Legendrian satellite construction \cite{MR2131643}.)
Similarly, $Z^\pm Z^\mp S^\pm$ and $S^\pm Z^\mp Z^\pm$ are Legendrian
isotopic. It follows that we can transpose consecutive $+$ letters in
a word while preserving Legendrian-isotopy class, and the same for
consecutive $-$ letters.
Thus two words with the same $z^+,z^-,s^+,s^-$ that begin with the
same sign correspond to Legendrian-isotopic knots.

\begin{figure}
\centering

\includegraphics
{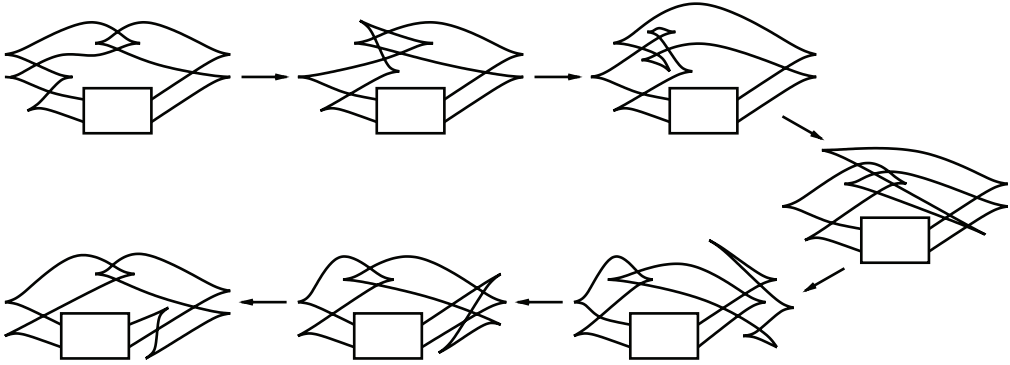}

\caption{
Moving a $Z$ from the beginning of a word to the end.
}
\label{fig:z-move}
\end{figure}

To complete the proof, it suffices to show that $Z^\pm w$ and $w
Z^\pm$ correspond to Legendrian isotopic knots for a length $2n-3$
word $w$, as do $S^\pm w$ and $w S^\pm$. For $Z^\pm w = w Z^\pm$, see
Figure~\ref{fig:z-move}; for $S^\pm w = w S^\pm$, reflect
Figure~\ref{fig:z-move} in the vertical axis.
\end{proof}

By Lemma~\ref{lem:legisotopic1}, we can define Legendrian isotopy
classes $\K_{z^+,z^-}$ for $0\leq z^\pm\leq n-1$ corresponding to words
with the specified $z^+,z^-$. We then have the following result.

\begin{lemma}\label{lem:legisotopic2}
The Legendrian isotopy classes $\K_{z^+,z^-}$ and $\K_{n-1-z^+,n-1-z^-}$
are the same.
\end{lemma}

\begin{proof}
The map $(x,y,z) \mapsto (-x,-y,z)$ is a contactomorphism of
$\mathbb{R}^3$ that preserves Legendrian isotopy classes, as can
easily be seen in the $xy$ projection, where it is a rotation by
$180^\circ$. In the $xz$ projection, this map sends tangles $Z^\pm$ to
$S^\pm$ and $S^\pm$ to $Z^\pm$ and thus sends $\K_{z^+,z^-}$ to
$\K_{n-1-z^+,n-1-z^-}$.
\end{proof}

We are now in a position to classify the $\K_{z^+,z^-}$'s and all
Legendrian knots obtained from the $\K_{z^+,z^-}$'s by
stabilization. The key ingredients are a result of Ozsv\'ath and
Stipsicz \cite{OzsvathStipsicz2008Pre} on distinct transverse representatives
of twist knots, and the ruling invariant discussed in Section~\ref{subsec:ruling}. Let $St^+,St^-$ denote the operations on Legendrian
isotopy classes given by positive and negative stabilization.

\begin{proposition}\label{prop:dist}
For $0\leq z^\pm,{z^\pm}'\leq n-1$, we have:
\begin{enumerate}
\item \label{class:1}
$\K_{z^+,z^-}$ is Legendrian isotopic to $\K_{{z^+}',{z^-}'}$ if and only if $({z^+}',{z^-}')
= (z^+,z^-)$ or $({z^+}',{z^-}')
= (n-1-z^+,n-1-z^-)$;
\item \label{class:2}
$\K_{z^+,z^-}$ and $\K_{{z^+}',{z^-}'}$ are Legendrian isotopic after some
positive number of positive stabilizations if and only if ${z^-}'=z^-$
or ${z^-}'=n-1-z^-$, and in these cases the knots are isotopic after one
positive stabilization;
\item \label{class:3}
$\K_{z^+,z^-}$ and $\K_{{z^+}',{z^-}'}$ are Legendrian isotopic after some
positive number of negative stabilizations if and only if ${z^+}'=z^+$
or ${z^+}'=n-1-z^+$, and in these cases the knots are isotopic after one
negative stabilization;
\item \label{class:4}
$St^+ St^- \K_{z^+,z^-}$ is Legendrian isotopic to $St^+ St^- \K_{{z^+}',{z^-}'}$ for all $z^\pm,{z^\pm}'$.
\end{enumerate}
\end{proposition}

\begin{figure}
\centering

\includegraphics
{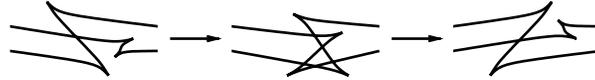}

\caption{
$Z$ and $S$ tangles are isotopic after an appropriate stabilization of each.}
\label{fig:szstab}
\end{figure}

\begin{proof}
We first establish (\ref{class:3}).
It is well-known  \cite{EpsteinFuchsMeyer01} that $Z^-$ and $S^-$ become Legendrian
isotopic after one negative stabilization; see
Figure~\ref{fig:szstab}. Consequently,
for $z^-<n-1$, $St^- \K_{z^+,z^-} = St^- \K_{z^+,z^-+1}$, and thus $St^-
\K_{z^+,z^-} = St^- \K_{z^+,{z^-}'} = St^- \K_{n-1-z^+,{z^-}''}$ for any
  $z^+,z^-,{z^-}',{z^-}''$, where the last equality follows from
  Lemma~\ref{lem:legisotopic2}. On the other hand, by
  \cite{OzsvathStipsicz2008Pre}, if $0\leq z^+,{z^+}'\leq n/2$ with $z^+ \neq
  {z^+}'$, then $\K_{z^+,z^+}$ and $\K_{{z^+}',{z^+}'}$ represent
  distinct Legendrian isotopy classes even after any number of
  negative stabilizations. (Note that $\K_{z^+,z^+}$ can be
  represented by the word $(Z^-Z^+)^{z^+} (S^-S^+)^{n-1-z^+}$, which
  corresponds to the Legendrian knot $E(2z^++1,2n-2z^+-1)$ in the
  notation of \cite{OzsvathStipsicz2008Pre}.)

 Item (\ref{class:3}) follows, and (\ref{class:2}) is proved
 similarly.
Item (\ref{class:4}) is an immediate consequence of (\ref{class:2})
and (\ref{class:3}), since stabilizations commute: $St^+ St^-
\K_{z^+,z^-} = St^+ St^- \K_{z^+,{z^-}'} = St^- St^+ \K_{z^+,{z^-}'} =
St^+ St^- \K_{{z^+}',{z^-}'}$.

\begin{figure}
\centering

\includegraphics
{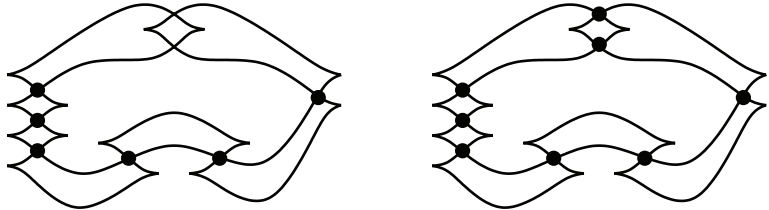}

\caption{
All possible $\rho$-graded rulings of the front for $\K_{z^+,z^-}$
(pictured here, $ZZSZ$ with either orientation). Dots indicate
switches. The left ruling is
$\rho$-graded for any $\rho$; all switches have Maslov degree $0$. The
two new switches in the right ruling have Maslov degree $2(z^++z^-+1-n)$
(top) and $-2(z^++z^-+1-n)$ (bottom), and thus the right ruling is
$\rho$-graded if and only if $\rho$ divides $2(z^++z^-+1-n)$.
}
\label{fig:ruling}
\end{figure}

It remains to establish (\ref{class:1}). The ``if'' part follows from
Lemma~\ref{lem:legisotopic2}. For ``only if'', we use
graded ruling invariants; one could also use Legendrian contact
homology \cite{Chekanov02}. The Maslov degrees of the two
uppermost (clasp) crossings in a
representative front diagram for $\K_{z^+,z^-}$ are readily seen to be
$\pm 2(z^++z^-+1-n)$. It follows from this that there is exactly one
$\rho$-graded (normal) ruling of the front unless $\rho \, | \, 2(z^++z^-+1-n)$,
in which case there are two $\rho$-graded rulings; See
Figure~\ref{fig:ruling}.

Now suppose that $\K_{z^+,z^-} = \K_{{z^+}',{z^-}'}$.
By Proposition~\ref{prop:ChP}, we must have
$|z^++z^-+1-n| = |{z^+}'+{z^-}'+1-n|$. On the other hand, by (\ref{class:2}) and
(\ref{class:3}), ${z^+}' \in \{z^+,n-1-z^+\}$ and ${z^-}' \in
\{z^-,n-1-z^-\}$. Combined, these equations imply that $({z^+}',{z^-}') =
(z^+,z^-)$ or $(n-1-z^+,n-1-z^-)$, as desired.
\end{proof}

We next consider $K_m$ when $m\leq -3$ is odd, say $m=-2n-1$; the argument here is similar to, but simpler than, the case of $m \leq -4$ even. According to Theorem~\ref{thm:frontneg}, we can represent each of the maximal-tb representatives of $K_{-2n-1}$ by a
length $2n-1$ word in the letters $Z^+,Z^-,S^+,S^-$, where these letters
represent the Legendrian half-twists shown in
Figure~\ref{fig:zs-example} and letters
must alternate in sign and begin and end with the same sign. The
Legendrian isotopy at the end of the proof of
Lemma~\ref{lem:legisotopic1} shows that we may assume that the word
begins (and ends) with a letter with a plus sign. As above, given such
a word $w$, let
$z^+(w),z^-(w),s^+(w),s^-(w)$ denote the number of $Z^+,Z^-,S^+,S^-$
in $w$, respectively, and note that $z^+(w)+s^+(w)=z^-(w)+s^-(w)+1=n.$

Essentially the same proof as for Lemma~\ref{lem:legisotopic1} gives the following result.
\begin{lemma}\label{lem:legisotopic1.1}
Two words of length $2n-1$ with the same $z^+,z^-,s^+,s^-$ correspond to
Legendrian-isotopic knots.\hfill\qed
\end{lemma}

By Lemma~\ref{lem:legisotopic1.1}, we can define Legendrian isotopy
classes $\K_{z^+,z^-}$ for $0\leq z^\pm\leq n$ corresponding to words
with the specified $z^+,z^-$. We then have the following result.
\begin{lemma}\label{lem:legisotopic2.1}
The Legendrian isotopy classes $\K_{z^+,z^-}$ and $\K_{n-z^+,n-1-z^-}$
are the same. If $z^+<n$ and $z^-\geq 1,$ then the Legendrian isotopy classes of $\K_{z^+,z^-}$ and $\K_{z^++1, z^--1}$ are the same.
\end{lemma}
\begin{proof}
The first statement follows as in the proof of Lemma~\ref{lem:legisotopic2}. For the second statement, let $S^+Z^-w'$ be a word representing $\K_{z^+,z^-}$, where $w'$ is some word of length $2n-3$; then $Z^+S^-w'$ represents $\K_{z^++1,z^--1}$. The isotopy in Figure~\ref{fig:z-move} shows that $Z^+S^-w'$ and $S^-w'Z^-$ correspond to Legendrian isotopic knots, while the reflection of this isotopy in the $z$ axis shows that $S^+Z^-w'$ and $Z^-w'S^-$ correspond to Legendrian isotopic knots. But $S^-w'Z^-$ and $Z^-w'S^-$ are also Legendrian isotopic by Lemma~\ref{lem:legisotopic1.1}.
\end{proof}
It follows from Theorem~\ref{thm:frontneg} and Lemma~\ref{lem:legisotopic2.1} that every maximal-tb knot has a representative of the form $\K_{n,z^-}$ for some $0 \leq z^-\leq n-1$; we denote this representative by $\K_{z^-}$.

\begin{proposition}\label{prop:dist2}
For $0\leq z^-,{z^-}'\leq n-1$, we have:
\begin{enumerate}
\item \label{class:1.1}
$\K_{z^-} $ is Legendrian isotopic to $\K_{{z^-}'}$ if and only if ${z^-}'= z^-$;
\item \label{class:2.1}
$St^\pm \K_{z^-}$ is Legendrian isotopic to  $St^\pm \K_{{z^-}'}$ for all $z^-,{z^-}'.$
\end{enumerate}
\end{proposition}
\begin{proof}
The proof of \eqref{class:2.1} is exactly the same as the proof of \eqref{class:2} and \eqref{class:3} in Proposition~\ref{prop:dist}.
For Item~\eqref{class:1.1} we again use $\rho$-graded rulings. As in the proof of Proposition~\ref{prop:dist}, all of the crossings in $\K_{z^-}$ have Maslov degree $0$, except for the top two crossings, which have grading $\pm (2z^-+1).$ So $K_{z^-}$ has one $\rho$-graded ruling unless $2z^-+1$ is divisible by $\rho$, in which case it has two. By Proposition~\ref{prop:ChP}, $\K_{z^-}$ and $\K_{{z^-}'}$ cannot be Legendrian isotopic unless $z^-={z^-}'.$
\end{proof}

We are now ready for the proof of our main theorem.
\begin{proof}[Proof of Theorem~\ref{thm:mtnrange}]
We begin with Items~\eqref{item1a} and~\eqref{item1b} of the theorem concerning the knot type $K_m$ with $m \geq 0$ (the case $m = -2$ is covered by \eqref{item3}).
Theorem~\ref{thm:frontpos} says that there is a unique Legendrian representative for $K_m$ with maximal Thurston--Bennequin number if orientations are ignored. For $m$ odd, when we take orientations into account, there are two maximal $\tb$ representatives $\K_+$ and $\K_-$ of $K_m$ and they are distinguished by their rotation numbers, $\rot(\K_\pm)=\pm 1$. Using the isotopy described in the proof of Lemma~\ref{lem:legisotopic2}, one easily verifies that $St^-(\K_+)$ is Legendrian isotopic to $St^+(\K_-).$ Since Theorem~\ref{thm:destab} says that all other representatives destabilize to $\K_\pm$, we conclude that $K_m$ is Legendrian simple if $m\geq 1$ is odd. When $m\geq 0$ is even, one may again use the isotopy described in the proof of Lemma~\ref{lem:legisotopic2} to check that the two oriented Legendrian representatives of $K_m$ coming from Theorem~\ref{thm:frontpos} are Legendrian isotopic. Thus there is a unique representative of $K_m$ with maximal Thurston--Bennequin n
 umber and all other Legendrian representatives are stabilizations of this one. This completes the proof for $m\geq 0$.

Next consider the case when $m$ is negative and even.
The maximal Thurston--Bennequin number representatives of $K_{-2n}$ are of the form $\K_{z^+,z^-}$ for $z^+,z^-\in \{0,\ldots, n-1\},$ by Theorem~\ref{thm:frontneg}. (This is even true after considering possible orientations, since the orientation reverse of $\K_{z^+,z^-}$ is $\K_{z^-,z^+}$.)
 Moreover, Proposition~\ref{prop:dist} says $\K_{z^+,z^-} = \K_{{z^+}',{z^-}'}$ if and only if $({z^+}',{z^-}')
= (z^+,z^-)$ or $({z^+}',{z^-}')
= (n-1-z^+,n-1-z^-).$ Since there are $n$ choices for $z^+$ and $z^-$ it is clear that there are $\lceil\frac{n^2}2\rceil$ distinct representatives. Similarly we see that after strictly positive of strictly negative stabilizations there are $\lceil\frac{n}2\rceil$ distinct representatives and after both type of stabilizations there is just one representative. This completes the proof of Item~\eqref{item3} of the theorem.

Similarly, when $m=-2n-1$ is negative and odd, Item~\eqref{class:1.1} in Proposition~\ref{prop:dist2} implies there are at least $n$ Legendrian representatives with maximal $\tb$, while Lemma~\ref{lem:legisotopic2.1} and the discussion around it implies there are at most $n$. Moreover Theorem~\ref{thm:destab} says all Legendrian representatives with non-maximal $\tb$ destabilize to one of these. Thus Item~\eqref{item2} of the theorem is completed by Item~\eqref{class:2.1} in Proposition~\ref{prop:dist2}.
\end{proof}

\begin{proof}[Proof of Theorem~\ref{main:transverse}]
We use the fact, due in this setting to \cite{EpsteinFuchsMeyer01}, that the negative stable classification of Legendrian knots is equivalent to the classification of transverse knots. More precisely, two transverse knots are transversely isotopic if and only if any of their Legendrian approximations are Legendrian isotopic after some number of negative stabilizations. Then Theorem~\ref{main:transverse} is a direct corollary of Theorem~\ref{thm:mtnrange}.
\end{proof}

\section{Normalizing the Front Projection}\label{sec:norm}
In this section we prove Theorems~\ref{thm:frontneg}, \ref{thm:frontpos}, and~\ref{thm:destab}, thus completing the proof of our main Theorem~\ref{thm:mtnrange}. To this end, notice that since $K_m$ is a rational knot, we can find an embedded $2$-sphere $S$ in $S^3 (=\R^3\cup\{\infty\})$ intersecting $K_m$ in four points and dividing $K_m$ into unknotted pieces. More precisely, we can choose $S$ as shown in Figure~\ref{fig:twistndec}, intersecting $K$ in four points labeled $1, 2, 3$ and $4$ in the figure and separating $S^3$ into two balls $\Bin$ and $\Bout$, such that: $K_m$ intersects $\Bout$ as a (vertical) 2-braid with two negative half-twists,
which we denote $K_{\textrm{out}}=K\cap \Bout$, and $K_m$ intersects $\Bin$ as a (horizontal) 2-braid with $m$ positive half-twists,
which we denote $K_{\textrm{in}}=K\cap\Bin.$
\begin{figure}[ht]
  \relabelbox \small {
  \centerline{\epsfbox{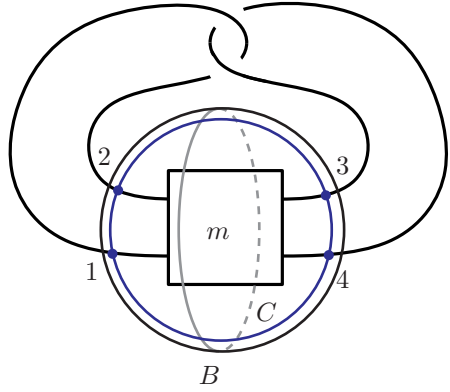}}}
  \relabel{1}{{$1$}}
  \relabel{2}{$2$}
  \relabel{3}{$3$}
  \relabel{4}{{$4$}}
  \relabel{b}{$B$}
  \relabel{c}{{$C$}}
  \relabel{m}{$m$}
  \endrelabelbox
\caption{Model of the knot $K_m$, intersecting a $2$-sphere $S$ in four points. The closed curve $C$ on $S$ intersects $K_m$ in four points and the closed curve $B$ on $S$ separates the points $1,2$ from $3,4$. 
}
\label{fig:twistndec}
\end{figure}

We begin by normalizing the dividing curves on $S$. After this we study the contact structures on the 3-balls $\Bin$ and $\Bout.$

\subsection{Normalizing the sphere $S$}

Throughout this section, we fix a standard model for $K_m$ as shown in Figure~\ref{fig:twistndec}, and we assume $m \neq -1$.
A Legendrian realization $\K$ of $K_m$ defines an isotopy $\psi\colon S^3\to S^3$ mapping $K_m$ to $\K$ and $S$ to $\psi(S).$ We can change the isotopy $\psi$ such that $\psi(S)$ is a convex surface, and a standard neighborhood $N$ of $\K$ with meridional ruling curves intersects $\psi(S)$ in four Legendrian unknots.
Let $P$ be the sphere with four punctures $P=S\setminus \nu(K_m)$.
The position of the pullback $\Gamma_P$ of the dividing curves on $\psi(P)$ depends on the chosen convex representation of $\psi(S)$, and thus on the isotopy $\psi$, but we can always choose $\psi$ so that $\Gamma_P$ is normalized as follows.

\begin{theorem}\label{thm:sphere}
Let $m \neq -1$, and fix $K_m$ along with a neighborhood $\nu(K_m)$ and the surfaces $S$ and $P$ as above.
For any Legendrian realization $\K$  of $K_m$, there exists an isotopy $\psi\colon S^3\to S^3$ such that $S$ (and thus $P$) is convex, $\psi(\nu(K_m))=N$ is a standard contact neighborhood of $\K$, and the pullback $\Gamma_P \subset P$ of the dividing curves on $\psi(P)$ is as shown in Figure \ref{fig:divcurve}.
\begin{figure}[ht]
  \relabelbox \small {
  \centerline{\epsfbox{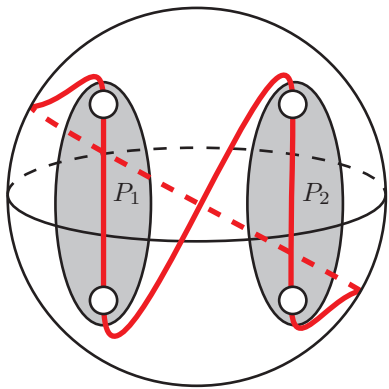}}}
   \relabel{1}{$P_1$}
  \relabel{2}{$P_2$}
  \endrelabelbox
\caption{The dividing curves on $P$ can always be arranged to be as shown (assuming $m\not=-1$).}
\label{fig:divcurve}
\end{figure}
\end{theorem}
Before proving Theorem~\ref{thm:sphere}, we establish the following lemma.
\begin{lemma}\label{lem:sphere}
If $m\not=-1$ and $\K$ is a Legendrian realization of $K_m,$ then there is a Legendrian unknot $\LL$ with $\tb(\LL)=-1$ that, in the complement of $\K,$ is topologically isotopic to the curve $B$ in Figure~\ref{fig:twistndec}.
\end{lemma}
\begin{proof}
There exists some Legendrian knot $\LL$ in the topological class of $B$, disjoint from $K_m$. Suppose that $\LL$ has been chosen so that $\tb(\LL)$ is maximal for Legendrians in this class, say $\tb(\LL)=-n$ for some $n>0$. We will show that the assumption that $n>1$ leads to a contradiction.

Let $N_\LL$ be a standard neighborhood of $\LL$ disjoint from $\K$. Set $Q=\overline{(S^3- N_\LL)}.$ Clearly $Q$ is a solid torus $S^1\times D^2$ with convex boundary, and the boundary has two dividing curves of slope $-n.$ We can assume the ruling curves are meridional and then choose two disks $D_1$ and $D_2$ in $Q$ bounding these ruling curves as shown in Figure~\ref{fig:constructS}.
\begin{figure}[ht]
  \relabelbox \small {
  \centerline{\epsfbox{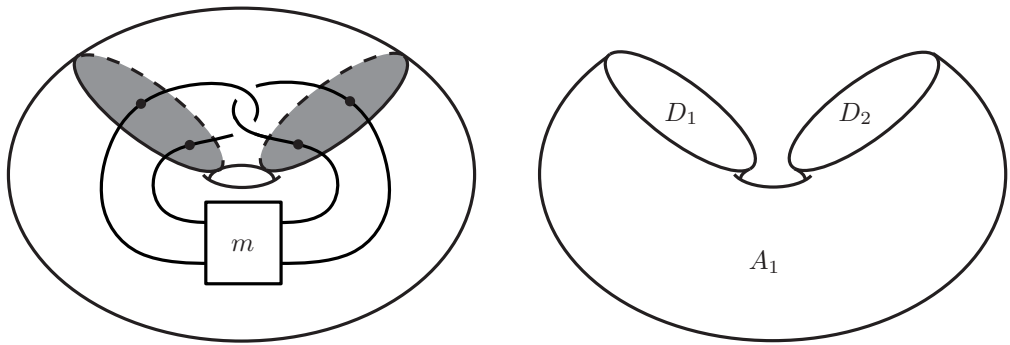}}}
  \relabel{m}{{$m$}}
  \relabel{1}{$D_1$}
  \relabel{2}{$D_2$}
  \relabel{a}{{$A_1$}}
  \endrelabelbox
\caption{The torus $\partial Q=\partial N_L$ on the left with the disk $D_1$ and $D_2$ shaded. On the right is the annulus $A_1$ and the disks $D_1$ and $D_2$ whose union can be taken to be $S.$}
\label{fig:constructS}
\end{figure}
Specifically, $\partial Q \setminus (\partial D_1\cup \partial D_2)$ consists of two annuli $A_1$ and $A_2$ such that $A_1\cup D_1\cup D_2$ (after rounding corners) represents the sphere $S$. We can isotop each $D_i$ so that a standard neighborhood $N$ of $\K$ intersects $D_i$ in two disks with Legendrian boundary (which are meridional ruling curves on $\partial N$) and $D_i$ is convex. Let $P_i=D_i\setminus N.$ Then $P_i$ is a pair of pants with three boundary components, which we label $c_{i,1},c_{i,2},c_{i,3}$ such that $c_{i,3}$ is the boundary component contained in $\partial Q$ and $c_{i,1},c_{i,2}$ are ruling curves in $\partial N$.

Notice that $\Gamma_{P_i}=\Gamma_P \cap P_i$ intersects each of $c_{i,1}$ and $c_{i,2}$ exactly twice and intersects $c_{i,3}$ exactly $2n$ times. If $\Gamma_{D_i}$ has more than two boundary parallel dividing curves then $\Gamma_{P_i}$ will have at least one boundary parallel dividing curve along $c_{i,3}$ and thus we can use this to construct a bypass to destabilize $\LL$ in the complement of $\K.$  As this is a contradiction, we know the dividing curves on $P_i$ can be described in the following way; see Figure~\ref{fig:P} for an illustration. There is a coordinate system on $D_i$ so that:
\begin{itemize}
\item
$D_i$ is the unit disk in the $xy$ plane;
\item
$\K\cap D_i$ is $\{(0,\pm 1/2)\}$;
\item
$C\cap D_i$ is the line segment $x=0$, where $C$ is as shown in Figure~\ref{fig:twistndec};
\item
$P_i$ is $D_i$ with small disks around $(0,\pm 1/2)$ removed.
\end{itemize}
Let $A_{i,j}$ be small annular neighborhoods of $c_{i,j}$ in $P_i$, and let $P_i'$ be the closure of the complement of these annuli in $P_i$. For $n>1$, the dividing curves on $P_i'$ can be assumed to be $n-2$ horizontal line segments, along with the line segments in $P_i'$ given by $y=\pm 1/2$. In addition, in each $A_{i,j}$, the dividing curves can be assumed to be the obvious extension of the dividing curves in $P_i'$, with some number of half-twists in $A_{i,1}$ and $A_{i,2}$ and some rigid rotation in $A_{i,3}$.
To elaborate on this last point, identify the closure of $A_{i,3}$ radially with $S^1\times[0,1]$ so that $\Gamma_{P_i}\cap (S^1\times \{0\})$ is $2n$ equally spaced points $p_1,\ldots,p_{2n}$ in the circle $c_{i,3}$ and $\Gamma_{P_i}\cap (S^1\times\{1\})$ is the corresponding set of $2n$ points $p_1',\ldots,p_{2n}'$ in the other boundary component of $A_{i,3}$; then in $A_{i,3}$, $\Gamma_{P_i}$ consists of $2n$ nonintersecting segments connecting $p_1,\ldots,p_{2n}$ to $p_1',\ldots,p_{2n}'$ in some (cyclically permuted) order.
\begin{figure}[ht]
  \relabelbox \small {
  \centerline{\epsfbox{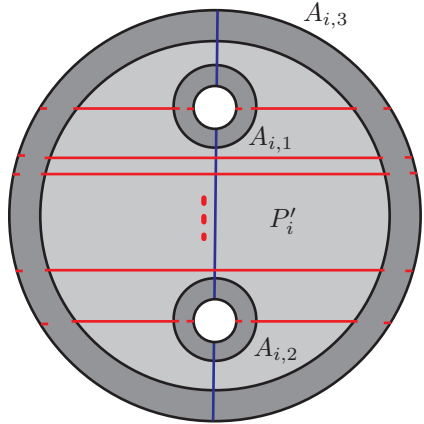}}}
  \relabel{p}{{$P'_i$}}
  \relabel{1}{$A_{i,1}$}
  \relabel{2}{$A_{i,2}$}
  \relabel{3}{{$A_{i,3}$}}
  \endrelabelbox
\caption{The disk $D_i.$ The lightly shaded region is $P_i'$ and the darkly shaded regions are the annuli $A_{i,j}$; the union of all the shaded regions is $P_i$; the vertical line is $C\cap P_i$; and the horizontal lines are the dividing curves $\Gamma_{P_i}.$ In the darkly shaded regions, the dividing curves cross from one boundary component to the other.}
\label{fig:P}
\end{figure}

After rounding the corners  of $D_1\cup D_2\cup A_1$, define $A$ to be the annulus $A_1\cup A_{1,3}\cup A_{2,3}.$ Notice that on $A$ there are $2n$ dividing curves running from one boundary component to the other (we know the dividing curves on $A_1$ as $A_1$ is part of $\partial N_\LL=\partial Q$); let $\Gamma_A$ denote the union of these dividing curves. As above we can choose a product structure $S^1\times [0,1]$ on the closure of $A$ so that $S^1$ has length $2n$, $\Gamma_A\cap (S^1\times \{0\})$ and $\Gamma_A\cap (S^1\times\{1\})$ each consists of $2n$ equally spaced points, and $\Gamma_A$ connects these two sets of points through $2n$ nonintersecting segments.
The dividing curve on the 2--sphere $S$ must be connected since we are in a tight contact structure, and thus the slope $s$ of the curves in $\Gamma_A$ must be relatively prime to $n.$

\begin{figure}[ht]
  \relabelbox \small {
  \centerline{\epsfbox{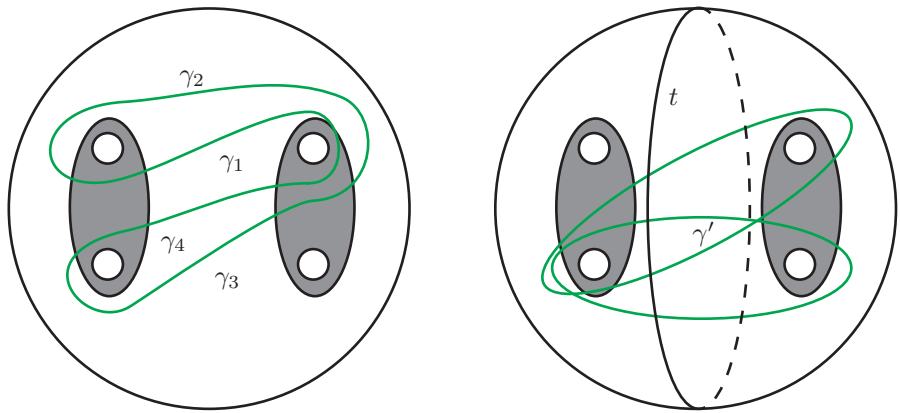}}}
  \relabel{1}{{$\gamma_1$}}
    \relabel{2}{{$\gamma_2$}}
      \relabel{3}{{$\gamma_3$}}
        \relabel{4}{{$\gamma_4$}}
  \relabel{p}{$\gamma'$}
  \relabel{t}{$t$}
  \endrelabelbox
\caption{The sphere $S$ with $P_1'$ and $P_2'$ shaded. The curve
  $\gamma$ is shown on the left and its intersection with the white
  annulus $A$ consists of the four curves
  $\gamma_1,\gamma_2,\gamma_3,\gamma_4.$
  If $m$ is even then $\gamma'$ is the image of the horizontal curve on
  the right after $\frac{m}{2}$ Dehn twists along the curve $t$ and if $m$ is
  odd then it is the image of the diagonal curve on the right after $\frac{m-1}{2}$
  Dehn twists along the curve $t$.}
\label{fig:G}
\end{figure}
Define curves $\gamma$ and $\gamma'$ in $S$ as shown in Figure~\ref{fig:G}. Then $\gamma$ and $\gamma'$
bound disks $\Dout$ in $\Bout$ and $\Din$ in $\Bin$, respectively,
where both disks are disjoint from $\K.$  We can assume that both
$\gamma$ and $\gamma'$ intersect the dividing curves of $S$ only in
$A$, and that the curve $\gamma$ intersects $A$ in four arcs
$\gamma_1,\gamma_2,\gamma_3,\gamma_4$ as shown in
Figure~\ref{fig:G}.  If $\gamma$ is isotoped so that it intersects the $P_i'$ in horizontal arcs, then
using the above identification of the closure of $A$ with $S^1\times [0,1]$, the slopes of
$\gamma_i$ can be taken to be $2, 0, n, n-2$, respectively.
Similarly $\gamma'$ intersects $A$ in two parallel linear arcs
$\gamma_1'$ and $\gamma_2'$ of slope $nm$.

Legendrian realize $\gamma$ and $\gamma'$, and make $\Dout$ and $\Din$ convex.
If $\gamma$ or $\gamma'$ does not have maximal $\tb$, then $\Dout$ or
$\Din$ has at least two boundary-parallel dividing curves and thus
there are at least two bypasses for $S\setminus N_\LL$ in the
complement of $\K.$ (Notice that when there are only two dividing curves
the two bypasses are not disjoint, however we will see below that
we will only need one bypass in this case, and in most cases.)
Let $c$ be the curve along which one of the
bypasses is attached. (Note that since $\gamma'$ bounds a disc in
$\Bin$ the bypass in that case is attached from the back so its action
on the dividing curves of $S$ is the mirror of
Figure~\ref{fig:bypassattach}.) We will show that in most cases the
bypass reduces $n$, leading to a contradiction. In particular, we have
the following claim.

\begin{claim}\label{claim}
If $c\cap (P_1'\cup P_2')$ has at most one component and $n\geq 2,$ then we can destabilize $\LL$ (contradicting the maximality of $\tb(\LL)$) except possibly when $n=3,$ in which case we can change $s$ by $1$ or $-1$ depending on whether $c$ is on $\gamma$ or $\gamma'.$
\end{claim}

\begin{remark}\label{remark:twice}
{\em In the proof below notice that when $n=3$ we can sometimes destabilize $\LL$ and sometimes change $s$. In the exceptional case when $\LL$ does not destabilize notice that $s$ must be relatively prime to $n$. Thus we can only attach such an exceptional bypass once and any subsequent bypasses attached from the same side of $S$, if it exists, cannot be exceptional and must then provide a destabilization of $\LL.$}
\end{remark}

We first prove the claim, then use it to complete the proof of the lemma.
\begin{proof}[Proof of Claim] First note that if $c\cap (P_1'\cup P_2')=\emptyset$, then when we attach the bypass to $A$, we see a destabilization for $\LL$ in the complement of $\K$, which contradicts the maximality of $\tb(\LL)$.
Thus to prove the claim, we may assume that $c\cap (P_1'\cup P_2')$ has one component. We treat the cases $n\geq 4$, $n=3$, and $n=2$ separately. For $n \geq 4$, there
are $8$ subcases shown in Figure~\ref{fig:cases}. The subcases for $\gamma'$ (when the bypass is attached from the back) are the mirrors of these cases.
\begin{figure}[ht]
  \relabelbox \small {
  \centerline{\epsfbox{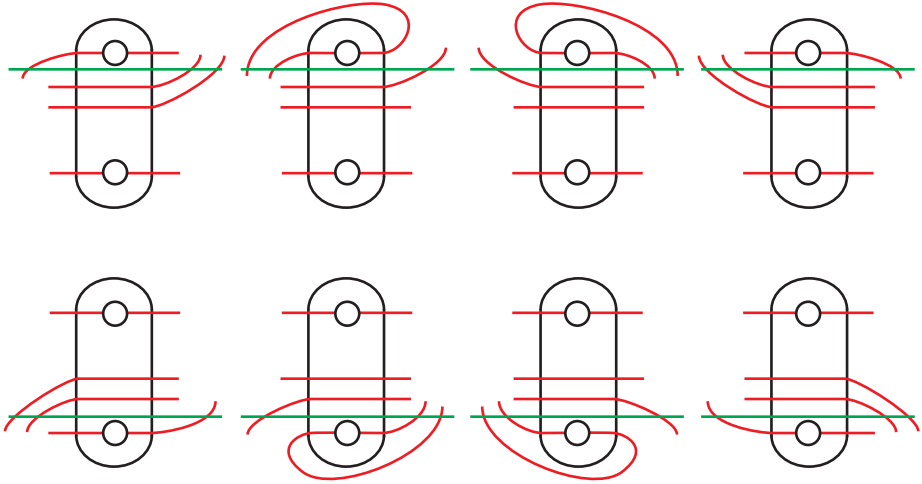}}}
  \endrelabelbox
\caption{The 8 subcases of Case 2 (ordered left to right and top to bottom).}
\label{fig:cases}
\end{figure}
In subcases 3, 4, 7, 8, one can use bypass rotation, Lemma~\ref{moveleft}, to obtain a bypass disjoint from $(P_1'\cup P_2')$ and hence destabilize $\LL.$  The bypasses in subcases 2 and 6 are disallowed: if there is a bypass there then we would have a convex sphere with disconnected dividing set, contradicting tightness.  In subcases 1 and 5 we can still destabilize $\LL$ if $n>3.$ Thus we contradict the minimality of $n$ in all subcases except when $n \leq 3$.

For $n=3$, we argue as above except for subcases 1 and 5. In these cases, notice that attaching the bypass does not destabilize $\LL$ but it does alter the dividing curves. Specifically pushing across the bypass adds (or subtracts, in the case of $\gamma'$) $1$ to the slope of $\Gamma_A$ once we have renormalized everything after attaching the bypass. Thus we see that when $n=3$ we can either destabilize $\LL$ or change the slope of $s$ by 1. This establishes the $n=3$ case of the claim.

Finally, for $n=2$, there are $4$ subcases analogous to Figure~\ref{fig:cases}. It is readily checked, as above, that two of these are disallowed by tightness, while the other two lead to destabilizations of $\LL$.
\end{proof}

We now return to the proof of the lemma. Since the statement of the
lemma is known for $m=0,1,\pm 2$ by the classification of Legendrian
unknots, torus knots, and the figure eight knot
\cite{EliashbergFraser99,EtnyreHonda01b}, we need only check it for
$|m|>2$. As mentioned in Remark~\ref{remark:twice}, there is an exceptional case when $s=3$ and the destabilization argument cannot be applied directly; we ignore this case for now and return to it at the end of the proof.

Since $\gamma'_1$ and $\gamma'_2$ are parallel, the intersection of $\gamma'$ with $\Gamma_{S \setminus N}$ must be essential.
It follows that $|\gamma_i'\cap\Gamma_A|=|s-nm|$ for $i=1,2$, since $\gamma_i'$ has slope $nm$ and $\Gamma_A$ has slope $s$. Now if
$|\gamma_i'\cap\Gamma_A| \geq 2$ for $i=1,2$, then for any bypass $c$ along $\gamma'$,
$c\cap (P_1'\cup P_2')$ has at most one component. Thus we can apply the claim if $|s-nm|\geq 2$. It is an easy exercise in algebra to check that $|s-nm|\geq 2$ for all $m$ with $|m|>2$ whenever $n\geq 2$ and $|s|\leq 3n-2$.
This establishes the lemma when $(n,s)$ is in the shaded region in the left diagram of Figure~\ref{fig:graph}.

\begin{figure}[ht]
  \relabelbox \small {
  \centerline{\epsfbox{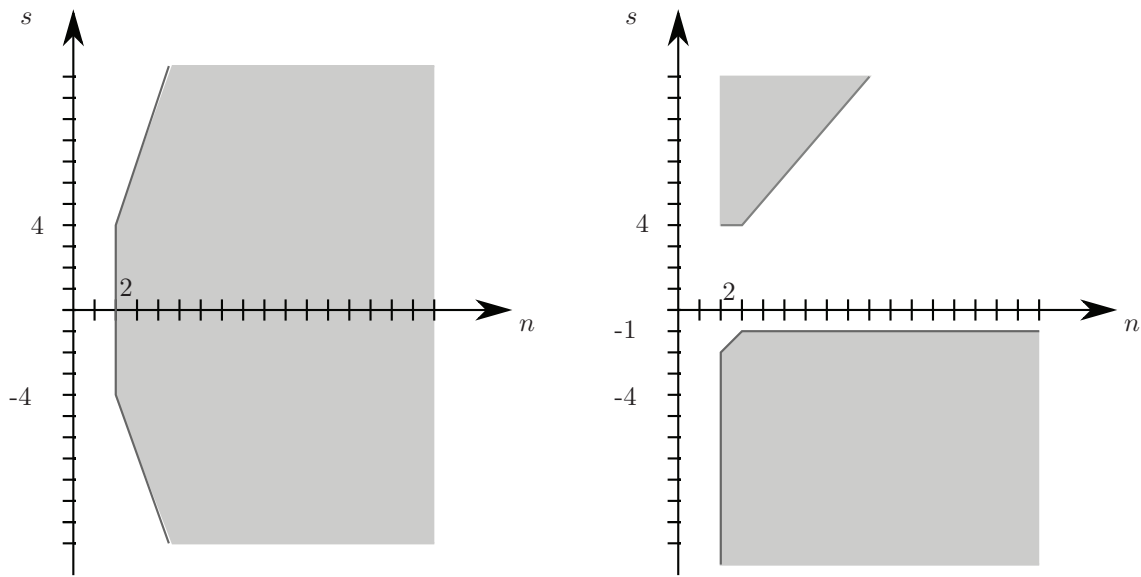}}}
  \relabel{nl}{{$s$}}
    \relabel{4l}{{4}}
      \relabel{m4l}{{-4}}
        \relabel{2sl}{{2}}
  \relabel{sl}{$n$}
\relabel{nr}{{$s$}}
    \relabel{4r}{{4}}
      \relabel{m1r}{{-1}}
\relabel{m2r}{{-4}}
\relabel{1sr}{{}}
        \relabel{2sr}{{2}}
  \relabel{sr}{$n$}
  \endrelabelbox
\caption{The pairs $(n,s)$ where the destabilization of $\gamma'$ gives a bypass satisfying the conditions of Claim \ref{claim}
(left), and the pairs $(n,s)$ where the destabilization of $\gamma$ 
gives a bypass satisfying the conditions of Claim \ref{claim}
(right).}
\label{fig:graph}
\end{figure}

Similarly, if $|\gamma_i\cap\Gamma_A|\geq 2$ for all but at most one
of $i=1,2,3,4$ and $|\gamma_i\cap\Gamma_A|\geq 1$ for the other $i$,
then we can apply the claim to at least one of the (at least two)
bypasses along $\gamma$. Now the intersection of $\gamma$ and
$\Gamma_{S\setminus N}$ is essential if the signs of $s-2,s,s-n,s-n+2$
all agree.  This is because in this case, a cancellation can only
occur around an arc going through one of the $P_i$'s, and there the
signs of the crossings agree with two of the above signs, so they
cannot cancel. Given that the intersection of $\gamma$ and
$\Gamma_{S\setminus N}$ is essential, we have
$|\gamma_1\cap\Gamma_A|=|s-2|$, $|\gamma_2\cap\Gamma_A|=|s|$,
$|\gamma_3\cap\Gamma_A|=|s-n|,$ and
$|\gamma_4\cap\Gamma_A|=|s-n+2|$. Thus we can apply the claim and
establish the lemma if the following conditions hold:
\begin{itemize}
\item
the signs of $s-2,s,s-n,s-n+2$ all agree;
\item
at least three of $|s-2|,|s|,|s-n|,|s-n+2|$ are $\geq 2$, and the fourth is $\geq 1$.
\end{itemize}
The set of $(n,s)$ for which these conditions hold is the shaded region in the right diagram of Figure~\ref{fig:graph}.

The union of the shaded regions in Figure~\ref{fig:graph} covers all of the half-plane $\{(n,s) : n\geq 2\}$. This covers all possible cases, and thus if $n\not=3$ the knot $\LL$ can always be destabilized by the claim, yielding the desired contradiction. When $n=3$ notice that we can always find two successive bypass attachments along arcs on $\gamma$ or $\gamma'$ that intersect $ (P_1'\cup P_2')$ at most one time.
(To see this notice that if there are not two bypasses along $\gamma'$ then from above we see that $s=\pm nm, \pm (nm\pm 1)$ or $\pm (nm\pm 2)$. Given that $|nm|\geq 6$ we see that in these cases we can find the two bypasses along $\gamma$.)
Thus, Remark~\ref{remark:twice} shows that $\LL$ can be destabilized in this case too.
\end{proof}

\begin{proof}[Proof of Theorem~\ref{thm:sphere}]
Throughout this proof we use the notation established at the beginning of this section and in the proof of Lemma~\ref{lem:sphere}. In particular notice that $P=P_1\cup P_2\cup A.$
We also set $P'=P_1'\cup P_2'\cup A,$ that is, $P'$ is $P$ with
annular neighborhoods of its boundary removed (in which there can be twisting of the dividing curves).

We begin by using Lemma~\ref{lem:sphere} to obtain a Legendrian unknot
$L$ with $\tb=-1$ as described in the lemma. We then use $L$ to create the pieces $P_1,P_2,A$ mentioned above. We will now analyze these pieces.

Recall that we have identified $A$ with $S^1\times [0,1]$ with $S^1$
having length 2, that is, $S^1=[0,2]/\sim$, where $\sim$ identifies the
endpoints of the interval. We further arrange that the dividing curves
$\Gamma_A$ intersect the boundary of $A$ at $\{0,1\}\times \{0,1\}.$
In these coordinates the slope of $\Gamma_A$ is some integer. In
Figure~\ref{fig:sop} we show two examples, one on the left with slope
$0$, and one on the right with slope $1$. All other slopes can be obtained from one of these examples by applying some number of Dehn twists along a curve parallel to the boundary of $A.$
\begin{figure}[ht]
  \relabelbox \small {
  \centerline{\epsfbox{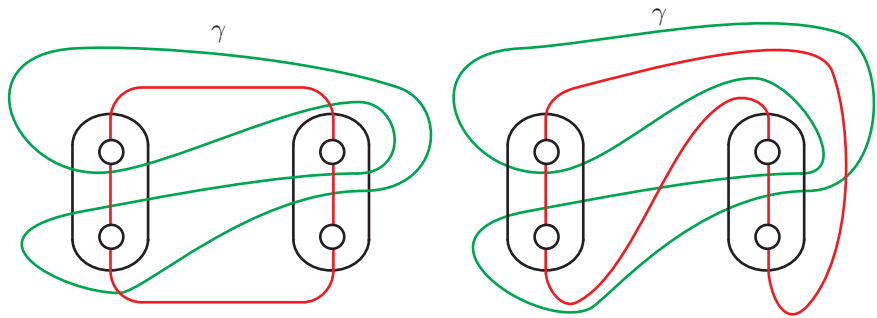}}}
  \relabel{1}{{$\gamma$}}
    \relabel{2}{{$\gamma$}}
  \endrelabelbox
\caption{On the left the dividing curves have slope $0$ on $A$, while on the right they have slope $1.$}
\label{fig:sop}
\end{figure}

As in the proof of Lemma~\ref{lem:sphere} we notice that there is
always a bypass for $P'$ on the disk $\Dout$ with boundary $\gamma.$
Suppose the bypass is attached to $P'$ along a curve $c.$ There are
three cases to consider: (1) $c$ is disjoint from $P_1\cup P_2$; (2)
it has an endpoint in $P_1$ or $P_2$; or (3) the center intersection
point of $c$ with $\Gamma_{P'}$ is contained in $P_1$ or $P_2.$
(Notice that the last two cases do not have to be disjoint if the
slope of $\Gamma_A$ is near zero.)
We consider further subcases depending on the slope of $\Gamma_A$,
which we denote by $s$.

If $s>2$, then
Case (1) results in a reduction of the slope by $2$, Case (2) results
in a reduction of the slope by $1$, and Case (3) is disallowed since
it results in a disconnected dividing curve on $S,$ contradicting the
tightness of the standard contact structure on $S^3$. Thus in this
case,
we can attach bypasses to $P$ to arrange that $s=2,1,$ or $0$.

If $s<-1$, then Cases (1) and (2) are disallowed, and Case (3)
increases the slope by 1. Thus in this case we can assume that $s=-1.$

We have now arranged that $s=2, 1, 0,$ or $-1.$ If $s=-1$, then there
are ten possible bypass attachments. Most are disallowed, while the
ones that are allowed can be used, after bypass rotation using
Lemma~\ref{moveleft}, to increase the slope of $\Gamma_A$ to $0$. If
$s=0$, then there are six possible places for a bypass along $\gamma$;
see the left hand side of Figure~\ref{fig:sop}. Of these, two give a disallowed bypass and the other four, after bypass rotation using Lemma~\ref{moveleft}, can be used to increase the slope of $\Gamma_A$ to $1.$

If $s=2$, then there are ten possible bypass attachments. Of these, four reduce the slope to $1$, and four are disallowed. The remaining two change the dividing curve on $P$ to the one shown in Figure~\ref{fig:2moresubcases} with $s=3$.
\begin{figure}[ht]
  \relabelbox \small {
  \centerline{\epsfbox{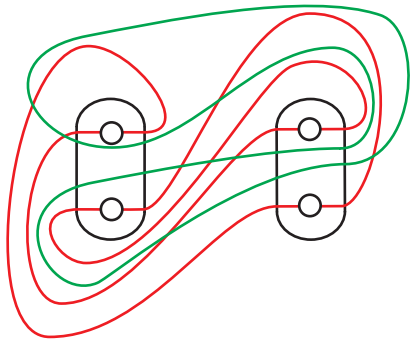}}}
  \endrelabelbox
\caption{The dividing curve after bypass attachment.}
\label{fig:2moresubcases}
\end{figure}
Examining the eight possible bypasses in this new situation, one sees
that six are disallowed and the remaining two return $\Gamma_P$ to the
configuration
with $s=1$.

We have proved that the dividing curves on $P'$ can be made to look
like those in Figure~\ref{fig:divcurve}. To complete the
proof of the theorem, notice that we have a standard neighborhood $N$
of $\K$ as claimed and the dividing curves on $P$ can differ from
those in the figure only by twisting in the annuli $A_{i,j}$ with $i,j=1,2.$ We can add a small neighborhood of the annuli $A_{i,j}$ to $N$ to get a new standard neighborhood of $\K$ (notice the slope of the dividing curves does not change as the neighborhood is increased to contain the annuli) and the surface $P'$ for the old neighborhood is the surface $P$ for the new neighborhood. This completes the proof of the theorem.
\end{proof}

\subsection{The contact structure in $\Bout$}
In this subsection we prove that either the Legendrian knot $\K$ destabilizes or the Legendrian tangle in $\Bout$ is determined.
\begin{theorem}\label{thm:bout}
Let $\K$ be a Legendrian knot in the knot type $K_m,$ $m\not= -1.$ Assume we have chosen an identification of $\K$ with the standard model as in Theorem~\ref{thm:sphere}. Let $D=\{x^2+z^2\leq 1\}\times \{y=0\}$ be a convex disk in $\R^3\subset S^3.$  Then either $\K$ destabilizes or there is a contactomorphism from $\Bout$ to the complement of (the interior of) a standard neighborhood $\{(x,y,z)\,|\, x^2+z^2\leq 1, y^2\leq 1\}$ of $D$ in $S^3$ (with corners rounded) taking $\K\cap \Bout$ to the curves shown in Figure~\ref{fig:bout}.
\end{theorem}

We notice that the dividing curves on the boundary of the ball in Figure~\ref{fig:bout} and the ones on $\Bout$ in Figure~\ref{fig:divcurve} are not the same but that there is a diffeomorphism of the ball that takes one set of curves to the other.
\begin{figure}[ht]
  \relabelbox \small {
  \centerline{\epsfbox{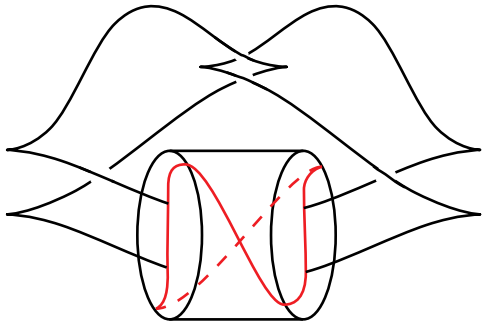}}}
  \endrelabelbox
\caption{Model for a non-destabilizable tangle in $\Bout.$ }
\label{fig:bout}
\end{figure}
\begin{proof}
Let $B$ be the complement of the interior of a neighborhood of $D$ in
$S^3$ with corners rounded and let $l_1$ (left) and $l_2$ (right) be
the two Legendrian arcs in $B$ shown in Figure~\ref{fig:bout}. Let
$N_1=D^2\times [0,1]$ and $N_2=D^2\times[0,1]$ be product
neighborhoods of $l_1$ and $l_2,$ respectively. We can assume each
$N_i\cap \partial B$ consists of two disks, each of which intersects
the dividing curve $\Gamma_{\partial B}$ in an arc. We can further
assume that the characteristic foliation of $\partial B$ has the
boundary of these disks as the union of leaves, and we can arrange that $\partial N_i$ is convex.

Notice that $H=B\setminus (N_1\cup N_2)$ is a genus $2$ handlebody whose boundary is a surface with corners. Two disks $D_1$ and $D_2$ that cut $H$ into a $3$-ball are indicated in Figure~\ref{fig:boutpd}. (To see where the disks come from notice that one can isotope the picture into a standardly embedded genus 2 handlebody in $S^3$ by isotoping the neighborhoods $N_1$ and $N_2$ so that they do not twist around each other. See Figure~\ref{fig:DisksForH}. 
\begin{figure}[ht]
  \relabelbox \small {
  \centerline{\epsfbox{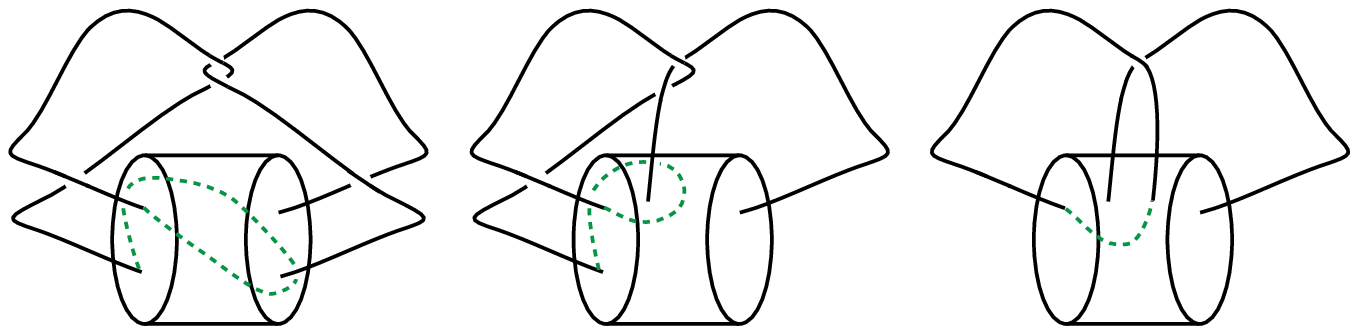}}}
  \endrelabelbox
\caption{On the right hand side we see there is an obvious disk bounded by the dotted curve and one of the arcs in $B$. The boundary of this disk is shown in the middle and left hand pictures after the arcs in $B$ are isotoped. On the right hand side we see the boundary of a disk that will provide a compressing disk for $H$. }
\label{fig:DisksForH}
\end{figure}
The disks are now obvious. Isotoping back to the original picture, once can track the boundaries of the disk to obtain Figure~\ref{fig:boutpd}.)
\begin{figure}[ht]
  \relabelbox \small {
  \centerline{\epsfbox{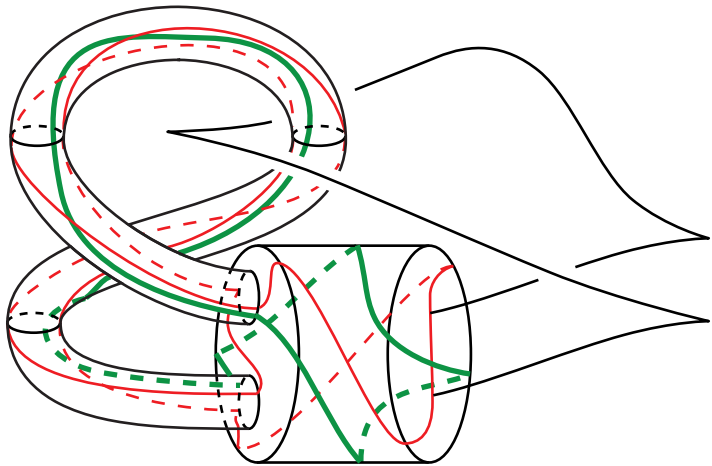}}}
  \endrelabelbox
\caption{The thicker curve bounds a disk $D_1$ in $H.$ The disk $D_2$
  can be seen by reflecting this picture about a vertical line. The
  thinner curves depict the dividing curves.}
\label{fig:boutpd}
\end{figure}

We now proceed to see what data determines the contact structure on $H$. 
A contact structure on a neighborhood $N_b$ of $\partial H$ is determined by the
characteristic foliation on $\partial H.$ We can find a slightly
smaller handlebody $\widehat{H}$ in $H$ such that $\partial
{\widehat{H}}$ is contained in $N_b$ and is obtained by rounding the
corners of $\partial H.$ Let $\widehat{D}_i=D_i\cap \widehat{H}$, and
Legendrian realize $L_i= \partial {\widehat{D}_i}$ on
$\partial{\widehat{H}}.$

Now $L_i$ intersects the dividing set $\Gamma_{\widehat{H}}$ in four
points, three of these points in (the part of $\partial{\widehat{H}}$
coming from) $\partial N_i$ (notice that
$\partial{\widehat{D}_i}\cap \partial N_i$ intersects the
dividing set efficiently, i.e., minimally in its homology class) and
one point $x$ in (the part of $\partial{\widehat{H}}$ coming from)
$\partial B.$ Thus we clearly have $\tb(L_i)=-2.$

Moreover, we can see that neither of the two boundary parallel
dividing curves on $\widehat{D}_i$ straddles the point $x$, as
follows. If one did, then the other dividing curve would give a bypass attached along
$\partial N_i.$ Attaching the bypass to $\partial \widehat{H}$ will result in a surface $\Sigma$
of genus two and a curve $\gamma$, which corresponds to $\partial \widehat{D}_i$ on $\partial\widehat{H}$. The curve $\gamma$ will intersect the dividing curves on $\Sigma$ twice. Compressing $\Sigma$ along a meridional disk to $N_{i+1}$ (where we use the convention that $N_3 = N_1$) will result in a convex torus $T$ on which $\gamma$ sits. (One may also think of $T$ as obtained by attaching the bypass to $\partial
(B\setminus N_{i+1})=\partial (H\cup N_{i+1})$.) The curve $\gamma$ is an essential curve in the torus $T$ and bounds a disk in the complement of $T$. Moreover, while it intersects the dividing set twice, it can be isotoped to be disjoint from it. Thus $\gamma$ can be Legendrian realized, resulting in an unknot with $\tb=0$ which contradicts tightness.


The dividing set on $\widehat{D}_i$ has now been completely
determined, and so the contact structure on $H$ is completely determined by the characteristic foliation on $\partial H$ (and after isotoping the boundary slightly, by $\Gamma_{\partial{H}}$).

We now turn our attention to $\K.$ We assume we have normalized $\K$, a neighborhood of $\K$, and the sphere $S$ as in Theorem~\ref{thm:sphere}. Let $l_1'$ and $l_2'$ be the Legendrian arcs that are the components of $\K\cap \Bout$ and $N'_1$ and $N'_2$ the components of $N\cap \Bout.$ The set $H'=\Bout\setminus(N'_1\cup N'_2)$ is a handlebody of genus 2 whose boundary is a surface with corners. We can choose disks $D_1'$ and $D_2'$ as shown in Figure~\ref{fig:boutr}. Notice that this figure differs from Figure~\ref{fig:bout} by a diffeomorphism of $\Bout$ and agrees with Figure~\ref{fig:twistndec} and the conclusion in Theorem~\ref{thm:sphere}.
\begin{figure}[ht]
  \relabelbox \small {
  \centerline{\epsfbox{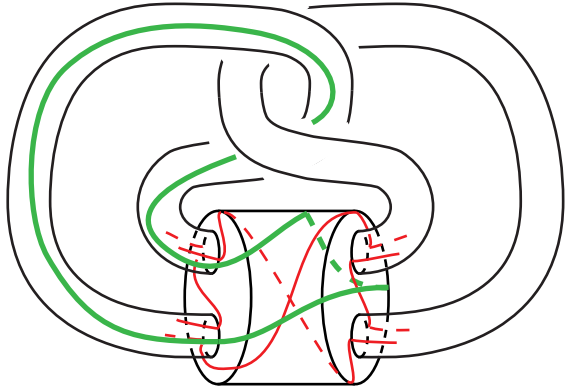}}}
  \endrelabelbox
\caption{The thicker curve bounds a disk $D'_1$ in $H'.$ The disk $D'_2$ can be seen by reflecting this picture about a vertical line.}
\label{fig:boutr}
\end{figure}
(We should take a neighborhood of $\partial H'$ and then take another
copy of the handlebody with the corners on the boundary rounded as we
did above, but for simplicity we will not include this in the
notation.) Legendrian realize $L_i'=\partial D_i'.$ We see that $L_i'$
intersects the dividing set on $\partial \Bout$ exactly twice, once
near $N_i'$ (and this intersection point can be assumed to be on
$N_i'$) and once at some point $y.$ While the intersection is
efficient, as above, if we consider $\partial (\Bout\setminus
N'_{i})$ then the intersection is inefficient. If $\tb(L_i)\leq -3$ then there is a bypass for $\partial N_i\cap \partial H'$ along $D'_i.$

Notice that $K_m$ bounds a singular disk $D'$ with a single clasp singularity.
\begin{claim}
The bypass above coming from $D'_i$
can be thought of as a bypass along $\partial D'$.
\end{claim}
Notice that the
framing given to $K_m$ by $D'$ is $(-1)^m.$ From
Proposition~\ref{prop:tbbound}, the maximal Thurston--Bennequin number
of $K_m$ is $\leq 1$ when $m$ is even and $\leq -1$ when $m$ is odd;
it follows that the contact framing on $\K$ is always less than or equal to the framing given by $D'.$ Thus a bypass along $\partial D'$ always gives a destabilization of $\K.$

To clarify the what the claim says, recall that the effect of attaching a bypass to a convex surface is entirely determined by its arc of attachment. Thus as long as the arc of attachment for a bypass on $D_i'$ is a subset of $\partial D'$ then we may assume it is a bypass on $D'$ as far as its effect on $\partial N$ is concerned. 

\begin{proof}[Proof of Claim]
Notice that $\partial D'$ can be broken into two parts $c_1=\partial D'\cap N_i$ and $c_2=(\partial D')\setminus c_1.$ Moreover we can assume that $c_1=\partial D_i'\cap N_i.$
If $c_1$ intersects the dividing curves on $\partial N_i$ efficiently
then with the appropriate orientations on $c_1$ and the dividing
curves, all the intersections between $c_1$ and the dividing curves
are negative, since if not then we could add a neighborhood of a
Legendrian arc in $\Bin$ to $N_i$ to construct a neighborhood of a
Legendrian unknot with nonnegative Thurston--Bennequin number,
contradicting tightness. We can similarly assume $c_2$ negatively intersects the
dividing curves on $N\setminus N_i$, where the orientation on the
dividing curves and $c_2$ are consistent with the orientations chosen
for the dividing curves on $N_i$ and $c_1$. We have shown that we can
arrange that $\partial D'$ intersects the dividing curves on $\partial
N$ efficiently and $\partial D'\cap N_i=\partial D_i'\cap N_i.$ Hence
any bypass along $\partial D_i'$ for $\partial N_i$ can be thought to be a bypass
attached along $\partial D'.$
\end{proof}

We assume for the remainder of this proof that $\K$ does not
destabilize. It follows from the above discussion that
$\tb(L'_i)=-2$ or $-1$. Arguing as we did for the standard model above, we see
that $\tb(L'_i)$ cannot equal $-1$, so we may assume
$\tb(L'_i)=-2$. Moreover, as above, we see that neither of the
two boundary parallel dividing curves on $D_i'$ can straddle $y.$ Thus
the configuration of the dividing curves $\Gamma_{D_i'}$ is
determined, and the contact structure on
$H'$ is determined by the characteristic foliation on $\partial H'.$

We can thus find a diffeomorphism
$\phi:B\to\Bout$ that preserves the dividing set on the
boundary and takes $l_i$ to $l_i'$ and $D_i$ to $D_i'.$ Since $\phi$ can
be isotoped to be a contactomorphism in a neighborhood of $(\partial
B)\cup l_1\cup l_2$ and the dividing curves on $D_i$ and $D_i'$
are determined above, we can isotop $\phi$ to be a contactomorphism
from $B$ to $\Bout$ taking $l_1\cup l_2$ to $l_1'\cup l_2'=\K\cap
\Bout.$
\end{proof}

\subsection{The contact structure in $\Bin$ and
  braids}\label{subsec:openbraids}
The disc $D=\{x^2+z^2\le1\}\times\{y_0\}$ is convex in $(S^3,\xist)$ with dividing curve $\Gamma_D=\{0\}\times [-1,1]\times\{y_0\}$. Fix $m$ points $\{(-1+\frac{2i}{m+1},y_0,0)\}_{i=1}^{m}$ on $\Gamma_D$. Then the fronts of Legendrian braids in $D\times [-1,1]$ with endpoints
$\{(-1+\frac{2i}{m+1},-1,0)\}_{i=1}^{m} \cup
\{(-1+\frac{2i}{m+1},1,0)\}_{i=1}^{m}$ are described as follows:

\begin{theorem}[Etnyre and V\'ertesi
\cite{EtnyreVertesi09Pre}]\label{thm:braids}
The Legendrian representations of a braid in $D\times [-1,1]$ are built up from the building blocks of Figure \ref{fig:openbraid}. \qed
\end{theorem}

\begin{figure}[ht]
  \relabelbox \small {
  \centerline{\epsfbox{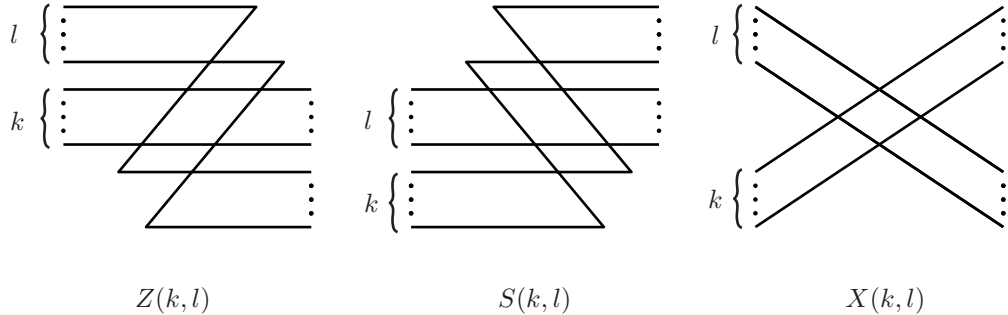}}}
  \relabel{1}{{$l$}}
  \relabel{2}{$k$}
  \relabel{3}{$l$}
  \relabel{4}{{$k$}}
  \relabel{5}{$l$}
  \relabel{6}{$k$}
  \relabel{Z}{{$Z(k,l)$}}
  \relabel{S}{$S(k,l)$}
  \relabel{X}{$X(k,l)$}
  \endrelabelbox
\caption{Building blocks of Legendrian braids. There can be other
  horizontal strands, not depicted, above and/or below the strands shown..}
\label{fig:openbraid}
\end{figure}

Note that $Z(0,1)$ and $S(1,0)$ are just stabilizations.
Using Theorem \ref{thm:braids} we can understand Legendrian braids
with two strands;
we will write $Z=Z(1,1)$, $X=X(1,1)$, $S=S(1,1)$.
Notice that if $Z$ or $S$ is followed by $X$ or vice versa then it
destabilizes, see Figure \ref{fig:SX}. This observation immediately
yields the following result.

\begin{figure}[ht]
  \relabelbox \small {
  \centerline{\epsfbox{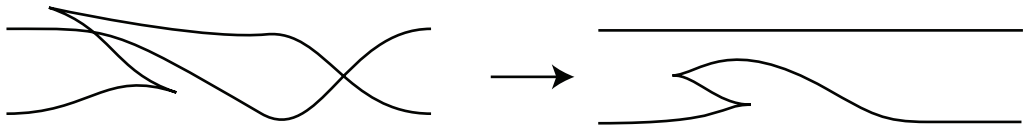}}}
  \endrelabelbox
\caption{$S=S(1,1)$ followed by $X=X(1,1)$ is Legendrian isotopic to a trivial braid with one stabilization.}
\label{fig:SX}
\end{figure}

\begin{proposition}\label{prop:leg2braid} Consider a braid with two strands and $n$ half twists.
\begin{enumerate}
\item If $n\ge 0$ then a  Legendrian representation of $B$ either destabilizes or consists of $n$ blocks of type $X$;
\item If $n<0$ then a Legendrian representation of $B$ either destabilizes or is built up from $n$ building blocks of type $S$ and $Z$ in any order. \qed
\end{enumerate}
\end{proposition}

This proposition allows us to understand $\K\cap \Bin.$
\begin{theorem}
Let $\K$ be a Legendrian knot in the knot types $K_m, m\not=-1.$ Either $\K$ destabilizes or
the contactomorphism from $\Bout$ to a ball in $S^3$ given in
Theorem~\ref{thm:bout} can be extended to $\Bin$, giving a
contactomorphism from $S^3$ to itself that maps $\K\cap \Bin$ to a Legendrian braid on two strands with $m+2$ twists.
\end{theorem}
\begin{proof}
The contactomorphism clearly extends as a diffeomorphism and since there is a unique contact structure up to isotopy on the 3-ball, we can isotop this diffeomorphism (relative to $\Bout$) to a contactomorphism on $\Bin.$ The image of $\K\cap \Bin$ is clearly a Legendrian 2-braid with $m+2$ twists.
\end{proof}

We are now ready to simultaneously prove  Theorems~\ref{thm:frontneg}, \ref{thm:frontpos}, and~\ref{thm:destab}.
\begin{proof}[Proof of  Theorems~\ref{thm:frontneg}, \ref{thm:frontpos}, and~\ref{thm:destab}]
If $m\not =-1$, let $\K$ be a Legendrian realization of $K_m.$ From
the previous theorem either $\K$ destabilizes or there is a
contactomorphism of $S^3$ taking $\K$ to one of the Legendrian knots
shown on the left of Figure~\ref{fig:legtwistn}; note that the box
shown there is a Legendrian 2-braid. If there is not an obvious
destabilization of the 2-braid, then by
Proposition~\ref{prop:leg2braid}, it is obtained by stacking $|m+2|$
$S$'s and $Z$'s together if $m\leq -2$, or $m+2$  $X$'s
if $m\geq 0.$ Clearly this agrees with Figure~\ref{fig:legtwistn} for
$m\leq -2,$ but also notice that for $m\geq 0$ this gives a knot
isotopic to the one in Figure~\ref{fig:mpos}. Since Legendrian isotopy
in the standard contact structure on $S^3$ is the same as ambient contactomorphism (i.e., a contactomorphism sending one Legendrian
knot to the other) \cite{Eliashberg91},
this completes the proof once we know that the knots shown in Figures~\ref{fig:legtwistn} and~\ref{fig:mpos} do not destabilize. But this is the content of Proposition~\ref{prop:tbbound}.
\end{proof}

\def\cprime{$'$} \def\cprime{$'$}


\end{document}